\documentclass[10pt,bezier]{article}
\usepackage{amsmath,amssymb,amsfonts,graphicx,caption,subcaption}
\usepackage[colorlinks,linkcolor=red,citecolor=blue]{hyperref}

\newtheorem{prethm}{{\bf Theorem}}[section]
\newenvironment{thm}{\begin{prethm}{\hspace{-0.5
em}{\bf.}}}{\end{prethm}}
\newtheorem{prepro}{{\bf Theorem}}

\newtheorem{precor}[prethm]{{\bf Corollary}}
\newenvironment{cor}{\begin{precor}{\hspace{-0.5
em}{\bf.}}}{\end{precor}}
\newtheorem{preconj}[prethm]{{\bf Conjecture}}
\newenvironment{conj}{\begin{preconj}{\hspace{-0.5
em}{\bf.}}}{\end{preconj}}
\newtheorem{preremark}[prethm]{{\bf Remark}}
\newenvironment{remark}{\begin{preremark}\em{\hspace{-0.5
em}{\bf.}}}{\end{preremark}}
\newtheorem{prelem}[prethm]{{\bf Lemma}}
\newenvironment{lem}{\begin{prelem}{\hspace{-0.5
em}{\bf.}}}{\end{prelem}}
\newtheorem{preque}[prethm]{{\bf Question}}

\newtheorem{preobserv}[prethm]{{\bf Observation}}

\newtheorem{predef}[prethm]{{\bf Definition}}

\newtheorem{preproposition}[prethm]{{\bf Proposition}}

\newtheorem{preproof}{{\bf Proof.}}
\newtheorem{preprooff}{{\bf Proof}}

\newenvironment{proof}[1]{\begin{preproof}{\rm
#1}\hfill{$\Box$}}{\end{preproof}}

\newtheorem{preproofs}{{\bf The second proof of }}

\newtheorem{preprooft}{{\bf Third proof of }}

\newtheorem{preproofF}{{\bf Proof of}}

\title{\bf\Large 
Modulo factors with bounded degrees
}
\author{{\normalsize{\sc Morteza Hasanvand${}$} }\vspace{3mm}
\\{\footnotesize{${}$\it Department of Mathematical
 Sciences, Sharif
University of Technology, Tehran, Iran}}
{\footnotesize{}}\\{\footnotesize{ $\mathsf{morteza.hasanvand@alum.sharif.edu }$ }}}

\date{}
\begin{document}
\maketitle
\begin{abstract}{
Let $G$ be a bipartite graph with bipartition $(X,Y)$, let $k$ be a positive integer, and let $f:V(G)\rightarrow \{-1,\ldots, k-2\}$ be a mapping with $\sum_{v\in X}f(v) \stackrel{k}{\equiv}\sum_{v\in Y}f(v)$. In this paper, we show that if $G$ is essentially $(3k-3)$-edge-connected and for each vertex $v$, $d_G(v)\ge 2k-1+f(v)$, then it admits a factor $H$ such that for each vertex $v$, $d_H(v)\stackrel{k}{\equiv} f(v)$, and $$\lfloor\frac{d_G(v)}{2}\rfloor-(k-1)\le d_{H}(v)\le \lceil\frac{d_G(v)}{2}\rceil+k-1.$$ Next, we generalize this result to general graphs and derive sufficient conditions for a highly edge-connected general graph $G$ to have a factor $H$ such that for each vertex $v$, $d_H(v)\in \{f(v),f(v)+k\}$. Finally, we show that every $(4k-1)$-edge-connected essentially $(6k-7)$-edge-connected graph admits a bipartite factor whose degrees are positive and divisible by $k$. 
} {\small
}
\end{abstract}
%
%
%
%
\section{Introduction}
In this article, graphs may have loops and multiple edges.
Let $G$ be a graph. The vertex set, the edge set, and the minimum degree of the vertices of $G$ are denoted by $V(G)$, $E(G)$, and $\delta(G)$, respectively.
We denote by $d_G(v)$ the degree of a vertex $v$ in the graph $G$.
If $G$ has an orientation, the out-degree and in-degree of $v$ are denoted by $d_G^+(v)$ and $d_G^-(v)$.
For a vertex set $A$ of $G$ with at least two vertices, the number of edges of $G$ with exactly one end in $A$ is denoted by $d_G(A)$.
Also, we denote by $e_G(A)$ the number of edges with both ends in $A$ 
and denote by $d_G(A,B)$ the number of edges with one end in $A$ and one end in $B$, where $B$ is a vertex set.
We denote by $G[A]$ the induced subgraph of $G$ with the vertex set $A$ containing
precisely those edges of $G$ whose ends lie in $A$, and denote by $G[A, B]$ the induced bipartite factor of $G$
with the bipartition $(A, B)$.
A graph $G$ is called 
{\it $m$-tree-connected} if it contains $m$ edge-disjoint spanning trees. 
Note that by the result of Nash-Williams~\cite{Nash-Williams-1961} and Tutte~\cite{Tutte-1961} every $2m$-edge-connected graph is $m$-tree-connected. In fact, the result of them says that a loopless graph $G$ is $m$-tree-connected if and only if 
$e_G(P)\ge m(|P|-1)$ for every partition $P$ of $V(G)$, where $e_G(P)$ denotes the number of edges of $G$ joining different parts of $P$.
A graph is termed 
{\it essentially $\lambda$-edge-connected},
 if all edges of any edge cut of size strictly less than $\lambda$ are incident with a common vertex.
Let $k$ be a positive integer. 
The cyclic group of order $k$ is denoted by $\mathbb{Z}_k$.
For two integers $x$ and $y$, we say that $x\stackrel{k}{\equiv}y$, if $x-y$ is divisible by $k$.
For any integer $n$, we denote by $[n]_k$ the unique integer $n_0$ such that $n_0\stackrel{k}{\equiv}n$ and 
$n_0\in \{-1,0,\ldots, k-2\}$.
An orientation of $G$ is said to be 
{\it $p$-orientation} if for each vertex $v$, $d_G^+(v)\stackrel{k}{\equiv}p(v) $,
 where $p:V(G)\rightarrow \mathbb{Z}_k$ is a mapping.
Likewise, an {\it $f$-factor} refers to a spanning subgraph $H$ such that for each vertex $v$, 
$d_H(v)\stackrel{k}{\equiv}f(v) $, where $f:V(G)\rightarrow \mathbb{Z}_k$.
For a graph $G$, we say that a mapping $f:V(G)\rightarrow \mathbb{Z}_k$ is 
{\it compatible with $G$} with respect to a bipartition $X,Y$ of $V(G)$ if 
$\sum_{v\in X} f(v)-2x\stackrel{k}{\equiv} \sum_{v\in Y}f(v)-2y$ for two integers $x$ with $0\le x\le e_G(X)$ and $0\le y\le e_G(Y)$. Note that one of $x$ and $y$ can be zero.
Likewise, we say that a mapping $f$ is 
{\it compatible with $G$} if it is compatible with $G$ with respect to every bipartition $X,Y$ of $V(G)$.
We will show that when $G$ is a $(2k-3)$-edge-connected bipartite graph with bipartition $(X,Y)$, $f$ is compatible with $G$ if and only if $f$ is compatible with respect to the bipartition $X,Y$ of $V(G)$; see Theorem~\ref{thm:compatible:sufficient}.
It is easy to see that if $G$ has an $f$-factor, then $f$ must be compatible with $G$.
The {\it bipartite index $bi(G)$} of a graph $G$ is the smallest number of all $|E(G)\setminus E(H)|$ taken over all bipartite factors~$H$.
A {\it modulo $k$-regular} graph refers to a graph whose degrees are positive and divisible by $k$. 
A proper coloring of a graph $G$ is to assign colors to vertices such that any two adjacent vertices receive different colors. The {\it chromatic number} $\chi(G)$ refers to the minimum number of necessary colors among all proper colorings.
Throughout this article, all variables $m$ are nonnegative integers and all variables $k$ are positive integers.

In 2008 Shirazi and Verstra{\"e}te~\cite{Shirazi-Verstraete-2008} introduced the concept of modulo factors and established a sufficient condition for the existence of $f$-factors modulo $k$ in graphs with average degree $2k-2$.

\begin{thm}{\rm(\cite{Shirazi-Verstraete-2008})}\label{thm:intro:Shirazi-Verstraete-2008}
{Let $G$ be a graph, let $k$ be a prime number, and let $f:V(G)\rightarrow \mathbb{Z}_k$ be an arbitrary mapping. 
If the number of orientations with out-degrees $k-1$ is not divisible by $k$, then $G$ has an $f$-factor.
}\end{thm}

In 2014 Thomassen formulated the following result about the existence of $f$-factors modulo $k$ in $(3k-3)$-edge-connected bipartite graphs. In this paper, we improve Theorem~\ref{Inro:factor:thm:Thomassen} by giving a sharp bound on degrees as mentioned in the abstract. Our result is based on a new improvement of the main result in~\cite{Lovasz-Thomassen-Wu-Zhang-2013} about the existence of modulo orientations with bounded out-degrees~\cite{ModuloBounded}.
\begin{thm}{\rm(\cite{Thomassen-2014})}\label{Inro:factor:thm:Thomassen}
{Let $G$ be a bipartite graph with bipartition $(X,Y)$, let $k$ be an integer, $k\ge 2$, and let $f:V(G)\rightarrow \mathbb{Z}_k$ be a mapping with $\sum_{v\in X}f(v) \stackrel{k}{\equiv}\sum_{v\in Y}f(v)$. 
If $G$ is $(3k-3)$-edge-connected, then it has an $f$-factor.
}\end{thm}

In 2016 Thomassen, Wu, and Zhang~\cite{Thomassen-Wu-Zhang-2016} generalized Theorem~\ref{Inro:factor:thm:Thomassen} to $(6k-7)$-edge-connected graphs with bipartite index at least $k-1$ provided that $k$ is odd. In Subsection~\ref{subsec:General graphs}, we extend their result to even integers $k$ and improve it to the following bounded-degree version. 

\begin{thm}\label{intro:thm:non-bipartite:6k-7}
{Let $G$ be a graph, let $k$ be a positive integer, and let $f:V(G)\rightarrow \mathbb{Z}_k$ be a mapping.
 Assume that $bi(G)\ge k-1$ and $(k-1)\sum_{v\in V(G)}f(v)$ is even.
If $G$ is $(6k-7)$-edge-connected, then it has an $f$-factor $H$ such that for each vertex~$v$,
$$\lfloor \frac{d_G(v)}{2}\rfloor-(k-1) \le d_H(v) \le \lfloor \frac{d_G(v)}{2}\rfloor+k.$$
}\end{thm}

As an application, we derive a sufficient condition for a highly edge-connected graph $G$ to have a factor whose degrees fall in predetermined integer sets as the following corollary. In Subsection~\ref{subsec:Improving-degree-bounds}, we 
 refine degree bounds for graphs with higher edge-connectivity.
\begin{cor}
{Let $G$ be a graph, let $k$ be a positive integer, and let $f$ be a positive integer-valued function on $V(G)$ satisfying
 $f(v) \le \frac{1}{2}d_G(v)< f(v)+k$ for each vertex $v$. 
Assume that $bi(G)\ge k-1$ and $(k-1)\sum_{v\in V(G)}f(v)$ is even.
If $G$ is $(6k-7)$-edge-connected, then it has a factor $H$ such that for each vertex $v$,
$$d_H(v)\in \{f(v),f(v)+k\}.$$
}\end{cor}

In~\cite{Thomassen-2014}, Thomassen established a sufficient edge-connectivity condition for the existence of a special type of bipartite modulo $k$-regular factors by concluding the following result from Theorem~\ref{Inro:factor:thm:Thomassen}.
In Section~\ref{sec:modulo regular}, we push down the needed edge-connectivity to $10k-3$ in essentially $(12k-7)$-edge-connected graphs. 
Moreover, we show that every $(4k-1)$-edge-connected essentially $(6k-7)$-edge-connected graph admits a bipartite modulo $k$-regular factor.
\begin{thm}{\rm (\cite{Thomassen-2014})}\label{thm:intro:2k-7}
{Every $(12k-7)$-edge-connected graph of even order has a bipartite modulo $k$-regular factor whose degrees are not divisible by $2k$. 
}\end{thm}

In 1984 Alon, Friedland, and Kalai proposed the following elegant conjecture and confirmed it for the case that $k$ is a prime power.
\begin{conj}{\rm (\cite{Alon-Friedland-Kalai-1984})}\label{intro:conj}
{Let $G$ be a loopless graph of order $n$ and let $k$ be a positive integer.
If $|E(G)|> (k-1)n$, then $G$ admits a modulo $k$-regular subgraph.
}\end{conj} 

Recently, Botler, Colucci, and Kohayakawa (2020)~\cite[Lemma 3]{Botler-Colucci-Kohayakawa} proved a weaker version of this conjecture by replacing the lower bound by $(24k-12)n$ based on Theorem~\ref{thm:intro:2k-7}, and applied it for their purpose. In this paper, we introduce a simpler technique to improve their result by replacing a better lower bound less than $4(k-1)n$.
\section{Basic tools and preliminary results}
\subsection{Tools: Orientations modulo $k$}
For making some results on the existence of modulo factors with bounded degree, we need to apply some results on the existence of modulo orientations with bounded out-degrees.
\begin{thm}{\rm (\cite{ModuloBounded}, see also Theorem 3.1 in \cite{Lovasz-Thomassen-Wu-Zhang-2013})}\label{thm:new-lower-bound}
{Let $G$ be a loopless graph with $z_0\in V(G)$, let $k$ be an integer, $k\ge 3$, and let $p:V(G)\rightarrow \mathbb{Z}_k$ be a mapping with
$|E(G)|\stackrel{k}{\equiv} \sum_{v\in V(G)}p(v)$. Let $D_{z_0}$ 
be an orientation of the set of edges incident with $z_0$.
 Assume that $G$ is essentially $(3k-3)$-edge-connected, and
\begin{enumerate}{

\item 
 $d_G(z_0)\le 2k-1+p(z_0)$, and the edges incident with $z_0$ are directed such that $d^+_G(z_0)\stackrel{k}{\equiv}p(z_0)$.

\item
$d_G(v)\ge 2k-1+[p(v)]_k$, for each $v\in V(G)\setminus \{ z_0\}$.
}\end{enumerate}

Then the orientation $D_{z_0}$ can be extended to a $p$-orientation $D$ of $G$ such that for each vertex~$v$,
$$\lfloor \frac{d_G(v)}{2}\rfloor - (k-1)\le\; d^+_G(v) \; \le \lceil \frac{d_G(v)}{2}\rceil+(k-1).
$$
}\end{thm}
\begin{cor}\label{cor:new-lower-bound}
{Let $G$ be a loopless graph with $z_0\in V(G)$, let $k$ be an integer, $k\ge 3$, and let $p:V(G)\rightarrow \mathbb{Z}_k$ be a mapping with
$|E(G)|\stackrel{k}{\equiv} \sum_{v\in V(G)}p(v)$. Let $D_{z_0}$ 
be an orientation of the set of edges incident with $z_0$. 
Let $s$ and $s_0$ be two nonnegative integer-valued functions on $V(G)$ satisfying $s_0(v)+s(v)< 2k$ for each vertex $v$, and $s(z_0)=s_0(z_0)=0$.
 Assume that $G$ is essentially $(3k-3)$-edge-connected, and
\begin{enumerate}{

\item 
 $d_G(z_0)+\sum_{v\in V(G)}\max\{s(v), s_0(v)\}< 2k$, and the edges incident with $z_0$ are directed such that $d^+_G(z_0)\stackrel{k}{\equiv}p(z_0)$.

\item
$d_G(v)\ge 2k-1+[p(v)]_k$, for each $v\in V(G)\setminus \{ z_0\}$.
}\end{enumerate}

Then the orientation $D_{z_0}$ can be extended to a $p$-orientation $D$ of $G$ such that for each vertex~$v$,
$$\lfloor \frac{d_G(v)+s(v)}{2}\rfloor - (k-1)\le\; d^+_G(v) \; \le \lceil \frac{d_G(v)-s_0(v)}{2}\rceil+(k-1).
$$
}\end{cor}
\begin{proof}
{Define $S $ to be the set of all $v\in V(G)\setminus \{z_0\}$ such that there is an integer $q(v)$ satisfying 
$q(v)\stackrel{k}{\equiv}p(v)$ and $d_G(v)/2 < q(v)\le d_G(v)/2+s(v)/2$. 
Likewise, define $S_0$ to be the set of all $v\in V(G)\setminus \{z_0\}$ such that 
there is an integer $q(v)$ satisfying $q(v)\stackrel{k}{\equiv}p(v)$ and $d_G(v)/2 -s_0(v)/2 \le q(v)< d_G(v)/2$.
Since $s(v)+s_0(v)< 2k$, we must have $S\cap S_0=\emptyset$.
Let $G'$ be the graph obtained from $G$ by 
adding $|2q(v)-d_G(v)|$ new parallel edges $vz_0$ for all $v\in S\cup S_0$.
We orient these new edges $vz_0$ toward $v$ when $v\in S$, and orient them toward $z_0$ when $v\in S_0$.
Define $p'(v)=p(v)$ for each $v\in V(G)\setminus (S_0\cup \{z_0\})$, 
$p'(v)=p(v)+|2q(v)-d_G(v)|$ for each $v\in S_0$, and 
$p'(z_0)=p(z_0)+\sum_{v\in S}|2q(v)-d_G(v)|$.
It is easy to check that $|E(G')|\stackrel{k}{\equiv}\sum_{v\in V(G')}p'(v)$, and
 $d_{G'}(v)\ge 2k-1+[p'(v)]_k$ for each $v\in V(G)\setminus\{z_0\}$, and also $d_{G'}(z_0)=d_{G}(z_0)+
\sum_{v\in S\cup S_0}|2q(v)-d_G(v)|\le d_{G}(z_0)+\sum_{v\in V(G)}\max\{s(v), s_0(v)\}\le
 2k-1\le 2k-1+p'(z_0)$.
Therefore, by Theorem~\ref{thm:new-lower-bound}, the graph $G'$ has a $p'$-orientation modulo $k$ such that for each vertex $v$, 
$|d^+_{G'}(v)-d_{G'}(v)/2|<k$. 
Since for each $v\in S\cup S_0$, $p'(v)\stackrel{k}{\equiv}d_{G'}(v)/2$, we must have $d^+_{G'}(v)=d_{G'}(v)/2$.
Thus this orientation induces a $p$-orientation for $G$ such that for each $v\in S\cup S_0$, $d^+_{G}(v)=q(v)$.
In addition, $|d^+_{G}(v)-d_{G}(v)/2|<k$ for each $v\in V(G)\setminus (S\cup S_0)$. 
According to the definition of $S$ and $S_0$, we must therefore have $d_G(v)/2 -s(v)/2-k < d^+_G(v) < d_G(v)/2 -s_0(v)/2+k$ for all $v\in V(G)$.
Hence the proof is completed.
}\end{proof}
\begin{cor}{\rm (Corollary 3.3 in \cite{ModuloBounded})}\label{Orientation:3k-3:cor:essentially:edge}
{Let $G$ be a loopless graph, let $k$ be an integer, $k\ge 3$, let $p:V(G)\rightarrow \mathbb{Z}_k$ be a mapping with 
$|E(G)|\stackrel{k}{\equiv}\sum_{v\in V(G)}p(v)$.
If $G$ is essentially $(3k-3)$-edge-connected and for each vertex $v$, $d_G(v)\ge 2k-1+[p(v)]_k$, then $G$ has a $p$-orientation such that for each vertex $v$,
$$\lfloor \frac{d_G(v)}{2}\rfloor - (k-1)\le\; d^+_G(v) \; \le \lceil \frac{d_G(v)}{2}\rceil+(k-1).
$$
Furthermore, for an arbitrary vertex $z$, $d^+_{G}(z)$
can be assigned to any plausible integer value in whose interval.
}\end{cor}
\begin{proof}
{It is enough to apply Corollary~\ref{cor:new-lower-bound} 
on the graph obtained from $G$ by adding a new artificial vertex $z_0$ with degree zero.
For the desired restriction on $d^+_{G}(z)$, we only need to set $\{s(z),s_0(z)\}=\{0,2k-1\}$ and $s(v)=s_0(v)=0$ for all vertices with $v$ with $v\neq z$.
}\end{proof}
\subsection{Compatible mappings}
\label{subsec:compatible}
In the following, we shall provide some sufficient conditions for a mapping to be compatible. Before doing so, let us make the following lemma which can conclude that every $(2bi(G)+1)$-edge-connected graph $G$ has a unique bipartition $X,Y$ of $V(G)$ satisfying $e_G(X)+e_G(Y)=bi(G)$. In particular, every connected bipartite graph has a unique bipartition.
\begin{lem}\label{lem:unique-bipartition}
{If $G$ is a $(2bi(G)+1+i)$-edge-connected graph and $i\ge 0$, then there is a unique bipartition $X,Y$ of $V(G)$ satisfying $e_G(X)+e_G(Y)\le bi(G)+i$.
}\end{lem}
\begin{proof}
{Let $X,Y$ be a bipartition of $V(G)$ satisfying $e_G(X)+e_G(Y)=bi(G)$.
Let $X',Y'$ be another bipartition of $V(G)$ with $\{X',Y'\}\neq \{X,Y\}$. Set $A=(X'\cap X)\cup (Y'\cap Y)$.
If $A$ is empty, then it is easy to check that $X'=Y$ and $Y'=X$. Likewise, if $V(G)\setminus A$ is empty, then we must have 
 $X'=X$ and $Y'=Y$. Thus $A$ is a nonempty proper subset of $V(G)$. 
Therefore, $e_G(X')+e_G(Y')\ge d_G(A)-(e_G(X)+e_G(Y))\ge (2bi(G)+1+i)-bi(G)>bi(G)+i$. Hence the proof is completed.
}\end{proof}
The following lemma introduces a lower bound on the bipartite index of $k$-tree-connected graphs.
\begin{lem}\label{lem:bi-lowebound}
{Let $G$ be a graph.
If $G[X,Y]$ is $k$-tree-connected and $e_G(X)+e_G(Y)\ge k$ for a bipartition $X, Y$ of $V(G)$, then $bi(G)\ge k$.
}\end{lem}
\begin{proof}
{Let $X',Y'$ be a partition of $V(G)$ with $e_G(X')+e_G(Y') = bi(G)$. 
Let $T_1,\ldots, T_k$ be $k$ edge-disjoint spanning trees of $G[X,Y]$ and let $e_1,\ldots, e_k$ be $k$ distinct edges of the graph $G[X]\cup G[Y]$. Since $T_i$ is a bipartite graph with the bipartition $(X,Y)$, the graph $T_i+e_i$ must contain an odd cycle $C_i$ which implies that $e_{C}(X')+e_{C}(Y')\ge 1$.
Therefore, $bi(G)=e_G(X')+e_G(Y') \ge \sum_{1\le i\le k}(e_{C_i}(X')+e_{C_i}(Y')) \ge k$.
}\end{proof}
The following theorem gives sufficient conditions for a mapping to be compatible with the main graph.
\begin{thm}\label{thm:compatible:sufficient}
{Let $G$ be a graph, let $k$ be a positive integer, and let $f:V(G)\rightarrow \mathbb{Z}_k$ be a mapping with $(k-1)\sum_{v\in V(G)}f(v)$ even.
Then $f$ is compatible with $G$ if one of the following conditions holds: 
\begin{enumerate}{
\item $k$ is even and $bi(G)\ge k/2-1$. 
\item $k$ is odd and $bi(G)\ge k-1$; see \cite{Thomassen-Wu-Zhang-2016}.

\item $G[X,Y]$ is $(2k-2)$-edge-connected and $e_G(X)+e_G(Y)\ge k-1$ for a bipartition $X, Y$ of $V(G)$.

\item $G$ is $(2k-3)$-edge-connected and $f$ is compatible with $G$ with respect to a bipartition $X,Y $ of $V(G)$ satisfying $e_G(X)+e_G(Y)\le k-2$.
}\end{enumerate}
}\end{thm}
\begin{proof}
{Let $X,Y$ be a bipartition of $V(G)$. 
First assume that $k$ is even and $e_G(X)+e_G(Y) \ge k/2-1$. 
By the assumption, $\sum_{v\in X}f(v)- \sum_{v\in Y}f(v)$ must be even.
Thus there are two integers $ x, y\in \{0,\ldots, k/2-1\}$ such that 
$\frac{1}{2}(\sum_{v\in X}f(v)-\sum_{v\in Y}f(v))\stackrel{k/2}{\equiv}x$ and $\frac{1}{2}(\sum_{v\in Y}f(v)-\sum_{v\in X}f(v))\stackrel{k/2}{\equiv}y$. Therefore, $x+y\stackrel{k/2}{\equiv}0$ which can conclude that $x \le e_G(X)$ or $y\le e_G(Y)$.
Now assume that $k$ is odd and $e_G(X)+e_G(Y) \ge k-1$. 
Let $ x, y\in \{0,\ldots, k-1\}$ be two integers such that $\sum_{v\in X}f(v)-\sum_{v\in Y}f(v)\stackrel{k}{\equiv}2x$ 
and $\sum_{v\in Y}f(v)-\sum_{v\in X}f(v)\stackrel{k}{\equiv}2y$.
Therefore, $x+y\stackrel{k}{\equiv}0$ which can conclude that $x \le e_G(X)$ or $y\le e_G(Y)$.
These imply that $f$ is compatible with $G$ with respect to $X,Y$. Hence the first two assertions hold.
To prove the third assertion, it is enough to apply Lemma~\ref{lem:bi-lowebound} together with the first two assertions.

Now, assume that $G$ is $(2k-3)$-edge-connected and $f$ is compatible with $G$ with respect to a bipartition $X,Y $ of $V(G)$ satisfying $ e_G(X)+e_G(Y)\le k-2=bi(G)+i$ and $i\ge 0$. Let $X',Y'$ be another bipartition of $V(G)$. Since $G$ is $(2bi(G)+1+i)$-edge-connected, by Lemma~\ref{lem:unique-bipartition}, we must have $e_G(X')+e_G(Y')\ge bi(G)+i+1\ge k-1$. Thus one can conclude that $f$ must be compatible with $G$ with respect to $X',Y'$ by repeating the proof of items (1) and (2).
This can complete the proof.
}\end{proof}
An immediate consequence of Theorem~\ref{thm:intro:Shirazi-Verstraete-2008} and the following corollary says that if $G$ is a graph satisfying $bi(G)\le k-2$, then the number of orientations with out-degrees $k-1$ 
 must be divisible by $k$, provided that $k$ is a prime number; for example, see~\cite[ Theorem 4]{Shirazi-Verstraete-2008}.
Recall that if a graph $G$ contains an $f$-factor modulo $k$, then $f$ must be compatible with $G$.
\begin{cor}
{Let $G$ be a graph, let $k$ be an integer with $k\ge 2$. Let $k_0=k$ when $k$ is odd, and let $k_0=k/2$ when $k$ is even.
Then $bi(G)\ge k_0-1$ if and only if every mapping $f:V(G)\rightarrow \mathbb{Z}_k$ with $(k-1)\sum_{v\in V(G)}f(v)$ even is compatible with $G$.
}\end{cor}
\begin{proof}
{Let $X, Y$ be a bipartition of $V(G)$ with $e_G(X)+e_G(Y)=bi(G)$. 
For a vertex $z\in X$, define $f(z)=2e_G(X)+2$ (mod $k$), and define $f(v)=0$ for all $v\in V(G)\setminus \{z\}$.
If $bi(G)<k_0-1$, then $f$ is not compatible with $G$.
Otherwise, there are two integers $x$ and $y$ with $0\le x\le e_G(X)$ and $0\le y\le e_G(Y)$ such that 
 $\sum_{v\in X} f(v)-2x\stackrel{k}{\equiv} \sum_{v\in Y}f(v)-2y$ which implies that $2(e_G(X)-x) +2y\stackrel{k}{\equiv}2(k-1)$. Therefore, $(e_G(X)-x) +y\stackrel{k_0}{\equiv}(k_0-1)$ and so $k_0 -1\le (e_G(X)-x) +y\le e_G(X)+e_G(Y)$, which is a contradiction.
Hence the proof can be completed using Theorem~\ref{thm:compatible:sufficient}.
}\end{proof}
\subsection{Bipartite index and tree-connectivity}
A well-known observation, attributed to Erd\H os (1965), says that every loopless graph with minimum degree at least $2m-1$ contains a bipartite factor with minimum degree at least $m$, see \cite[Theorem 2.4]{Bondy-Murty-2008}. This result is developed to an edge-connected version by Thomassen~(2008) as the following theorem.
\begin{thm}{\rm (\cite{Thomassen-2008-paths})}\label{thm:bipartite:factor:edge-version}
{Every $(2m-1)$-edge-connected graph has an $m$-edge-connected bipartite factor.
}\end{thm}
In the following, we shall provide a tree-connected version for it which will be used several times in this paper. This theorem is also generalized in \cite{B} for finding factors with bounded chromatic numbers.
\begin{thm}\label{thm:bipartite:factor}
{Every $2m$-tree-connected loopless graph $G$ has an $m$-tree-connected bipartite factor $H$ such that for every vertex set $A$,
$d_H(A)\ge \lceil d_G(A)/2 \rceil $.
}\end{thm}
\begin{proof}
{Let $H$ be a bipartite factor of $G$ with the maximum $|E(H)|$. 
We claim that $H$ is the desired factor.
Suppose, to the contrary, that $d_H(A)< d_G(A)/2 $ for a vertex set $A$.
Define $X_0=(X\setminus A)\cup (A\cap Y)$ and $Y_0=(Y\setminus A)\cup (A\cap X)$, where $(X,Y)$ is the bipartition of $H$. It is not difficult to see that the graph $G[X_0,Y_0]$ is a bipartite factor of $G$ with more edges than $H$ which is a contradiction.
Now, let $P$ be a partition of $V(G)$. Since $G$ is $2m$-tree-connected, we must have $e_G(P)\ge 2m(|P|-1)$, where
$e_G(P)$ denotes the number of edges of $G$ joining different parts of $G$.
Therefore, $e_H(P)=\sum_{A\in P}\frac{1}{2}d_H(A)\ge \sum_{A\in P}\frac{1}{4}d_G(A)=\frac{1}{2}e_G(P)\ge m(|P|-1).$
Hence by the well-known result of Nash-Williams~\cite{Nash-Williams-1961} and Tutte~\cite{Tutte-1961}, the graph $H$ must be $m$-tree-connected.
}\end{proof}
%
%
The following corollary gives a criterion for the existence of edge-disjoint odd cycles in $2k$-tree-connected graphs in terms of bipartite index. The edge-connected version of this corollary is mentioned in~\cite[Section 7]{Thomassen-2001-bi}.

\begin{cor}\label{cor:odd-cycles}
{Let $G$ be a $2k$-tree-connected graph.
Then $bi(G)\ge k$ if and only if $G$ contains $k$ edge-disjoint odd cycles.
}\end{cor}
\begin{proof}
{Let $X,Y$ be a partition of $V(G)$. If $C$ is an odd cycle, then obviously $e_{C}(X)+e_{C}(Y)\ge 1$.
This implies that $e_G(X)+e_G(Y)\ge k$ provided that $G$ contains $k$ edge-disjoint odd cycles.
Now, assume that $bi(G)\ge k$.
By Theorem~\ref{thm:bipartite:factor}, there exists a bipartition $X,Y$ of $V(G)$ such that 
$G[X,Y]$ is $k$-tree-connected. Since $e_G(X)+e_G(Y)\ge bi(G)\ge k$, by Lemma~\ref{lem:bi-lowebound}, the graph $G$ contains 
$k$ edge-disjoint odd cycles.
}\end{proof}
\begin{cor}
{Every $2k$-tree-connected graph $G$ satisfying $bi(G)\ge k$ contains a subgraph $H$ with maximum degree at most $2k$ satisfying $bi(H)\ge k$.
}\end{cor}
\begin{proof}
{By Corollary~\ref{cor:odd-cycles}, the graph $G$ contains $k$ edge-disjoint odd cycles which the union of them is the desired subgraph.
}\end{proof}
\begin{cor}\label{cor:decomposition:bipartite-index}
{Every $(2m+4)$-tree-connected graph $G$ can be decomposed into two factors $G_1$ and $G_2$ such that $G_1$ is Eulerian, $G_2[X,Y]$ is $m$-tree-connected for a bipartition $X,Y$ of $V(G)$, and
$$e_{G_2}(X)+e_{G_2}(Y) =\min\{k, bi(G)\},$$
where $k$ is an arbitrary nonnegative integer. 
}\end{cor}
\begin{proof}
{By Theorem~\ref{thm:bipartite:factor}, there exists a bipartition $X,Y$ of $V(G)$ such that 
$G[X,Y]$ is $(m+2)$-tree-connected. 
Decompose $G[X,Y]$ into two spanning trees $T_0$ and $T$ and 
an $m$-tree-connected factor $H$. Since $e_{G}(X)+e_{G}(Y) \ge bi(G)$, we can decompose $G[X] \cup G[Y]$ into two factors $M_0$ and $M$ such that $|E(M)|=\min\{k, bi(G)\}$. Let $F$ be a spanning forest of $T$ such that for each vertex $v$, $d_F(v)$ and $d_{T_0}(v)+d_{M_0}(v)$ have the same parity. It is enough to set $G_1=T_0\cup M_0\cup F$ and $G_2=G\setminus E(G_1)$ to complete the proof.
}\end{proof}
\begin{cor}
{Let $G$ be a graph with $bi(G)\ge k$.
If $G$ is $4k$-tree-connected, then it has $k$ edge-disjoint spanning Eulerian subgraphs with odd size.
}\end{cor}
\begin{proof}
{By Theorem~\ref{thm:bipartite:factor}, $G[X,Y]$ is $2k$-tree-connected for a partition $X,Y$ of $V(G)$.
Thus $G[X,Y]$ contains $2k$ edge-disjoint spanning trees $T_1,\ldots, T_k$ and $T'_1,\ldots, T'_k$.
By the assumption, $e_G(X)+e_G(Y)\ge k$. 
Thus we can take $e_1,\ldots, e_k$ to be $k$ edges of $G[X]\cup G[Y]$.
Define $H_i=T'_i+e_i$.
Let $F_i$ be a spanning forest of $T_i$ such that for each vertex $v$, $d_{F_i}(v)$ and $d_{H_i}(v)$ have the same parity.
It is enough to set $G_i=F_i\cup H_i$ to construct the desired Eulerian factors.
}\end{proof}
%
%
%
%
%
%
%
%
%
%
%
%
%
%
%
%
\section{Factors modulo $2$}
In this section, we consider the existence of parity factors. 
Our results are based on the following lemma which is a special case of a result due 
to Lov{\'a}sz (1970) who gave a criterion for the existence of parity $(g,f)$-factors. Here, we denote by $\omega (G)$ the number of components of a graph $G$, and a {\it parity $(g,f)$-factor} refers to a spanning subgraph $H$ such that for each vertex $v$, $g(v)\le d_H(v)\le f(v)$ and $d_H(v)\stackrel{2}{\equiv}g(v) \stackrel{2}{\equiv} f(v)$.
where $g$ and $f$ are two integer-valued functions on $V(G)$.
\begin{lem}{\rm(\cite{Lovasz-1972}; see also \cite[ Lemma 6.1]{Equitable-Factorization})}\label{lem:Lovasz1972}
{Let $G$ be a connected graph and let $g$ and $f$ be two integer-valued functions on $V(G)$ with $g\le f$ satisfying
 $\sum_{v\in V(G)}f(v)\stackrel{2}{\equiv}0$, and $f(v)\stackrel{2}{\equiv}g(v)$ for each vertex $v$.
Then $G$ has a parity $(g,f)$-factor,
if for any two disjoint subsets $A$ and $B$ of $V(G)$ with $A\cup B\neq \emptyset$,
$$\omega(G \setminus (A\cup B))\le 1+ \sum_{v\in A}f(v)+ \sum_{v\in B}(d_{G}(v)-g(v))-d_G(A,B).$$
}\end{lem}
It is known that edge-connectivity $1$ is sufficient for a graph to have an $f$-factor modulo~$2$, see~\cite{Edmonds-Johnson-1973, Thomassen-Wu-Zhang-2016}. The following theorem shows that edge-connectivity $2$ is sufficient for a graph to have an $f$-factor modulo~$2$ whose degrees fall in predetermined short intervals.
\begin{thm}\label{Factor:modulo2:thm:2:edge}
{Let $G$ be a graph and let $f:V(G)\rightarrow \mathbb{Z}_2$ be a mapping with
$\sum_{v\in V(G)}f(v) \stackrel{2}{\equiv} 0$. 
If $G$ is $2$-edge-connected, then it has an $f$-factor $H$ such that for each vertex~$v$,
$$\lfloor \frac{d_{G}(v)}{2}\rfloor-1\le d_H(v)\le \lceil \frac{d_{G}(v)}{2}\rceil+1.$$
Furthermore, for an arbitrary vertex $z$, $d_{H}(z)$ 
can be assigned to any plausible integer value in whose interval.
}\end{thm}
\begin{proof}
{For each vertex $v$, define $g'(v)\in \{\lfloor d_{G}(v)/2\rfloor-1, \lfloor d_{G}(v)/2\rfloor\}$ and
 $f'(v)\in \{ \lceil d_{G}(v)/2\rceil, \lceil d_{G}(v)/2\rceil+1\}$ such that 
$g'(v)\stackrel{2}{\equiv}f(v)\stackrel{2}{\equiv}f'(v)$. 
Obviously, $g'(v)\le f'(v)$. 
Note that if $g'(z)<f'(z)$, we allow to replace $g'(z)$ by $g'(z)+2$ or 
 replace $f'(z)$ by $f'(z)-2$ with respect to our purpose related to $z$.
For the first choice, we have $|d_G(z)/2-g'(z)|\le 3/2$ 
and for the second choice, we have $|d_G(z)/2-f'(z)|\le 3/2$.
Let $A$ and $B$ be two disjoint subsets of $V(G)$ with $A\cup B\neq \emptyset$.
 Since $G$ is $2$-edge-connected, it is not hard to check that
$$\omega(G\setminus (A\cup B)) \le \sum_{A\cup B}\frac{1}{2}d_G(v)-e_G(A\cup B) \le \sum_{A\cup B}\frac{1}{2}d_G(v)-d_G(A,B).$$
Therefore,
$$\omega(G\setminus (A\cup B))< 2+\sum_{v\in A}f'(v)+\sum_{v\in B}(d_G(v)-g'(v)) -d_{G}(A,B),$$
whether $z\in A\cup B$ or not.
Thus by Lemma~\ref{lem:Lovasz1972}, the graph $G$ has a parity $(g',f')$-factor and the proof is completed.
}\end{proof}
The following well-known result on Eulerian graphs is a special case of Theorem~\ref{Factor:modulo2:thm:2:edge}.
\begin{cor}\label{cor:Eulerian:1/2}
{Every Eulerian graph $G$ with $z\in V(G)$ has a factor $H$ 
such that for each~$v\in V(G)\setminus \{z\}$, $d_H(v)=d_G(v)/2$, and $d_H(z)\in \{d_G(z)/2, d_G(z)/2+1\}$.
}\end{cor}
\begin{proof}
{For each $v\in V(G)\setminus \{z\}$, define $f(v)=d_G(v)/2$ (mod $2$), and also
 define $f(z)=\sum_{v\in V(G)\setminus \{z\}}f(v)$ (mod $2$). 
Since $G$ is $2$-edge-connected, by Theorem~\ref{Factor:modulo2:thm:2:edge}, the graph $G$ has an $f$-factor $H$ such that for each~$v\in V(G)\setminus \{z\}$, 
$d_{G}(v)/2-1\le d_H(v)\le d_{G}(v)/2+1$, and 
$d_{G}(z)/2\le d_H(z)\le d_{G}(z)/2+1$.
Hence $H$ is the desired factor.
}\end{proof}
An interesting application of Theorem~\ref{Factor:modulo2:thm:2:edge} is given in the following corollary.
\begin{cor}
{Every connected $2r$-regular graph $G$ with $(r+1)|V(G)|$ even can be decomposed into two factors whose degrees lie in the set $\{r-1,r+1\}$.
}\end{cor}
\begin{proof}
{For each vertex $v$, define $f(v)=r+1$ (mod $2$). By Theorem~\ref{Factor:modulo2:thm:2:edge}, the graph $G$ has an $f$-factor such that for each vertex $v$, $r-1\le d_{G}(v)/2-1\le d_H(v)\le d_{G}(v)/2+1 \le r+1$.
Hence $H$ and its complement are the desired factors whose degrees lie in the set $\{r-1,r+1\}$.
}\end{proof}
\begin{remark}
{Note that Theorem~\ref{Factor:modulo2:thm:2:edge} can conclude Theorem 11 in~\cite{Bujtas-Jendrol-Tuza}. 
A generalization of it is formulated in \cite[Theorem 6.2]{Equitable-Factorization}.
}\end{remark}
%
%
%
%
%
%
%
%
%
In the following theorem, 
we develop Theorem~\ref{Factor:modulo2:thm:2:edge} to a partition-connected version.
A graph $G$ is said to be 
{\it $(m,l_0)$-partition-connected},
 if it can be decomposed into an $m$-tree-connected factor and a factor $F$ which admits an orientation such that for each vertex $v$, $d^+_F(v)\ge l_0(v)$,
where $l_0$ is a nonnegative integer-valued function on $V(G)$. 
\begin{thm}\label{Factor:modulo2:thm:2:partition}
{Let $G$ be a graph and let $f:V(G)\rightarrow \mathbb{Z}_2$ be a mapping with 
$ \sum_{v\in V(G)}f(v) \stackrel{2}{\equiv} 0$. 
Let $s$, $s_0$, and $l_0$ be three integer-valued functions on $V(G)$ satisfying $s(v)+s_0(v)< d_G(v)$ 
and $\max\{s(v),s_0(v)\}\le l_0(v)$ for each vertex $v$. 
If $G$ is $(1, l_0)$-partition-connected, then it has an $f$-factor $H$ such that for each vertex $v$, 
$$s(v)\le d_{H}(v)\le d_G(v)-s_0(v).$$
}\end{thm}
\begin{proof}
{For each vertex $v$, define $g'(v)\in \{s(v), s(v)+1\}$ and $f'(v)\in \{d_G(v)-s_0(v)-1, d_G(v)-s_0(v)\}$ such that 
$g'(v)\stackrel{2}{\equiv}f(v)\stackrel{2}{\equiv}f'(v)$. 
Note that the condition $s(v)+s_0(v)< d_G(v)$ implies that $g'(v)\le f'(v)+1$ and hence $g'(v)\le f'(v)$, because those have the same parity. 
By the assumption, the graph $G$ can be decomposed into two factors $T$ and $F$ such that $T$ is a spanning tree and $F$ admits an orientation such that for each vertex $v$, $d^+_F(v)\ge l_0(v)$.
Let $A$ and $B$ be two disjoint subsets of $V(G)$. It is not hard to check that
$$\omega(G\setminus (A\cup B)) \le \omega(T\setminus (A\cup B))= \sum_{A\cup B}(d_T(v)-1)+1-e_T(A\cup B)\le \sum_{A\cup B}(d_T(v)-1)+1-d_T(A,B).$$
Since $d^+_F(v)\ge l_0(v)\ge \max \{s(v),s_0(v)\}$ for each vertex $v$, we must have 
$$0\le \sum_{v\in A\cup B}d^-_F(v)-d_F(A,B)\le 
\sum_{v\in A}(d_F(v)-s_0(v))+\sum_{v\in B}(d_F(v)-s(v))-d_F(A,B).$$
 Therefore,
$$\omega(G\setminus (A\cup B))\le \sum_{v\in A}(d_G(v)-s_0(v)-1)+
\sum_{v\in B}(d_G(v)-s(v)-1) -d_{G}(A,B)+1,$$
which implies that 
$$\omega(G\setminus (A\cup B))\le \sum_{v\in A}f'(v)+
\sum_{v\in B}(d_G(v)-g'(v)) -d_{G}(A,B)+1.$$
Thus by Lemma~\ref{lem:Lovasz1972}, the graph $G$ has a parity $(g',f')$-factor and the proof is completed.
}\end{proof}
%
%
%
%
%
%
%
%
%
%
%
%
%
\section{Factors modulo $k$: Almost bipartite graphs}
\label{sec:Factors:edge-connected graphs}
\subsection{Bipartite graphs}

There is a special one-to-one mapping between orientations
and factors of any bipartite graph, which was
utilized
by Thomassen in~\cite{Thomassen-2014} in order to establish Theorem~\ref{Inro:factor:thm:Thomassen}.
Using the same arguments, we derive the following strengthened version. 

\begin{thm}\label{thm:essentially:factor:k}
{Let $G$ be a bipartite graph with bipartition $(X,Y)$, let $k$ be a positive integer,
and let $f:V(G)\rightarrow \mathbb{Z}_k$ be a mapping with $\sum_{v\in X}f(v) \stackrel{k}{\equiv}\sum_{v\in Y}f(v)$.
If $G$ is essentially $(3k-3)$-edge-connected and 
$d_G(v)\ge 2k-1+[f(v)]_k$ for each vertex $v$,
then $G$ has an $f$-factor $H$ such that for each vertex~$v$,
$$ \lfloor \frac{d_G(v)}{2}\rfloor-(k-1)
\le d_H(v) \le
 \lceil \frac{d_G(v)}{2}\rceil +(k-1).$$
Furthermore, for an arbitrary vertex $z$, $d_{H}(z)$ 
can be assigned to any plausible integer value in whose interval.
}\end{thm}
\begin{proof}
{The special case $k=2$ follows from Theorem~\ref{Factor:modulo2:thm:2:edge}. Assume $k\ge 3$.
For each $v\in X$, define $ p(v) = f(v)$, and for each $v\in Y$, define $ p(v) = d_G(v)-f(v)$.
Since $f$ is compatible with $G$, $\sum_{v\in X}f(v)\stackrel{k}{\equiv}\sum_{v\in Y}f(v)$ which can conclude that $|E(G)|\stackrel{k}{\equiv}\sum_{v\in V(G)}p(v)$. 
By the assumption for each $v\in X$, $d_G(v)\ge 2k-1+[f(v)]_k =2k-1+[p(v)]_k$.
Also, for each $v\in Y$, $d_G(v)\ge 2k-1+[f(v)]_k =2k-1+[d_G(v)-p(v)]_k$ which can imply that $d_G(v) \ge 2k-1+[p(v)]_k$.
More precisely, if we let $i = d_G(v)- 2k$, then since $i-[f(v)]_k \ge -1$, we must have $d_G(v) \ge 2k+i-[f(v)]_k-1 \ge 2k-1+[i-f(v)]_k=2k-1+[p(v)]_k$.
Thus by Corollary~\ref{Orientation:3k-3:cor:essentially:edge},
the graph $G$ has a $p$-orientation modulo $k$ 
such that for each vertex $v$, $ \lfloor d_G(v)/2\rfloor-(k-1)
\le d^+_G(v) \le
 \lceil d_G(v)/2\rceil +(k-1)$.
Take $H$ to be the factor of $G$ consisting of all edges directed from $X$ to $Y$.
Since for all vertices $v\in X$, $d_H(v)=d_G^+(v)$, 
and for all vertices $v\in Y$, $d_H(v)=d_G(v)-d^+_G(v)$, the graph $H$ is the desired $f$-factor.
The remaining case $k= 1$ can be proved similarly, because it is known that every graph $G$ has an orientation such that for each vertex $v$, $ \lfloor d_G(v)/2\rfloor \le d^+_G(v) \le \lceil d_G(v)/2\rceil$; in particular, we can arbitrarily have $d^+_G(z)=\lfloor d_G(z)/2\rfloor$ or $d^+_G(z)=\lceil d_G(z)/2\rceil$ by reversing the orientation (if necessary).
}\end{proof}
Bensmail, Merker, and Thomassen (2017)~\cite{Bensmail-Merker-Thomassen-2017}
applied a weaker version of the following corollary to deduce that every $16$-edge-connected bipartite graph
admits a decomposition into at most two locally irregular subgraphs. Their proof is based on Theorem 5.2 in~\cite{Bensmail-Merker-Thomassen-2017} for the special case $k=6$. By replacing the following result in their proof, this number can be pushed down to $15$.
\begin{cor}\label{cor:bipartite:factor}
{Let $G$ be a bipartite graph with bipartition $(X,Y)$, let $k$ be a positive integer, and let $f:V(G)\rightarrow \mathbb{Z}_k$ be a mapping with $\sum_{v\in X}f(v) \stackrel{k}{\equiv}\sum_{v\in Y}f(v)$. 
If $G$ is $(3k-3)$-edge-connected, then it has an $f$-factor $H$ such that for each vertex $v$,
$$ \lfloor \frac{d_G(v)}{2}\rfloor-(k-1)
\le d_H(v) \le
 \lceil \frac{d_G(v)}{2}\rceil +(k-1).$$
Furthermore, for an arbitrary vertex $z$, $d_{H}(z)$ 
can be assigned to any plausible integer value in whose interval.
}\end{cor}
\begin{proof}
{Apply Theorem~\ref{thm:essentially:factor:k} with the fact that for each vertex $v$, $d_G(v)\ge 3k-3 \ge 2k-1+[f(v)]_k$.
}\end{proof}
We have the following immediate conclusions similar to Corollaries~3.8 and~3.9 in \cite{ModuloBounded}.
Note the required edge-connectivity $3k-3$ of the first one can be replaced by odd edge-connectivity $3k-2$
and the second one can be replaced by the condition $d_G(A)\ge 6k-2$, where $A$ is a subset of $V(G)$ with odd size.
\begin{cor}
{Let $G$ be a bipartite graph and let $k$ be an odd positive integer.
If $G$ is $(3k-3)$-edge-connected, then it has a factor $H$ such that for each vertex $v$,
$$d_H(v)\in \{\frac{d_G(v)}{2}-\frac{k}{2},\; \frac{d_G(v)}{2},\; \frac{d_G(v)}{2}+\frac{k}{2}\}.$$
}\end{cor}
\begin{proof}
{For each vertex $v$, define $f(v)=d_G(v)/2$ (mod $k$) when $d_G(v)$ is even, 
and define $f(v)=(d_G(v)+k)/2$ (mod $k$) when $d_G(v)$ is odd.
Let $(X,Y)$ be the bipartition of $G$.
Obviously, $\sum_{v\in X}f(v)\stackrel{k}{\equiv}|E(G)|/2+n_xk/2$ and 
 $\sum_{v\in Y}f(v)\stackrel{k}{\equiv}|E(G)|/2+n_yk/2$,
where $n_x$ and $n_y$ are the number of vertices in $X$ and $Y$ with odd degrees, respectively.
Since $n_x$ and $n_y$ have the same parity, $\sum_{v\in X}f(v)\stackrel{k}{\equiv}\sum_{v\in Y}f(v)$.
Thus by Corollary~\ref{cor:bipartite:factor}, the graph $G$ has an $f$-factor such that for each vertex $v$,
$\lfloor d_G(v)/2\rfloor-(k-1)\le d_H(v) \le \lceil d_G(v)/2\rceil+(k-1)$.
If $d_G(v)$ is even, then $d_H(v)=d_G(v)/2$.
Otherwise, since $d_G(v)/2-3k/2 < d_H(v) < d_G(v)/2+3k/2$, we must have $d_H(v)\in \{d_G(v)/2-k/2,d_G(v)/2+k/2\}$.
Hence $H$ is the desired factor.
}\end{proof}
\begin{cor}\label{cor:Eulerian:even-k}
{Let $G$ be a bipartite Eulerian graph of even order and let $k$ be a positive integer.
If $G$ is $(6k-2)$-edge-connected, then it has a factor $H$ such that for each vertex $v$,
$$d_H(v)\in \{\frac{d_G(v)}{2}-k,\; \frac{d_G(v)}{2}+k\}.$$
}\end{cor}
\begin{proof}
{For each vertex $v$, define $f(v)=d_G(v)/2+k$ (mod $2k$), and let $(X,Y)$ be the bipartition of $G$. 
Obviously, $\sum_{v\in X}d_G(v)/2=|E(G)|/2= \sum_{v\in Y}d_G(v)/2$.
Since $|V(G)|$ is even, we must have $\sum_{v\in X}f(v)\stackrel{2k}{\equiv}\sum_{v\in Y}f(v)$.
Thus by Corollary~\ref{cor:bipartite:factor}, the graph $G$ has an $f$-factor $H$ such that for each vertex $v$,
$d_G(v)/2-3k <d_G(v)/2-(2k-1)\le d_H(v) \le d_G(v)/2+2k-1<d_G(v)/2+3k$.
Hence $H$ is the desired factor.
}\end{proof}
The following corollary is an improved version of Lemma 4.1 in~\cite{Merker-2017}.
Note the required edge-connectivity 3k-3 can be replaced by the condition $d_G(A)\ge 3k-3$, where $A$ is a subset of $V(G)$ 
 satisfying $0< |A\cap Y|< |Y|$, by replacing Corollary 3.8 in \cite{ModuloBounded} in the proof.
\begin{cor}
{Let $G$ be a $(3k-3)$-edge-connected bipartite graph on classes $X$ and $Y$, where each vertex in $X$ has even degree.
For every function $f:Y\rightarrow \mathbb{Z}_k$ satisfying
$\sum_{v\in Y}f(v)\stackrel{k}{\equiv}\frac{1}{2}|E(G)|$,
there exists a factor $H$ of $G$ such that 
\begin{enumerate}{
\item $d_{H}(v)=\frac{1}{2}d_G(v)$ for each $v\in X$.
\item $|d_H(v)- \frac{1}{2}d_G(v)|<k$ and 
$d_H(v)\stackrel{k}{\equiv} f(v)$ for each $v\in Y$.
}\end{enumerate}
}\end{cor}
\begin{proof}
{For each $v\in X$, define $f(v)=d_G(v)/2$ (mod $k$). According to the assumption, $\sum_{v\in X}f(v) \stackrel{k}{\equiv}\sum_{v\in X}d_G(v)/2=\frac{1}{2} |E(G)|\stackrel{k}{\equiv}\sum_{v\in Y}f(v)$.
Thus by Corollary~\ref{cor:bipartite:factor}, the graph $G$ has an $f$-factor $H$ such that for each vertex $v$,
$ d_G(v)/2+k <\lfloor d_G(v)/2 \rfloor-(k-1)
\le d_H(v) \le \lceil d_G(v)/2\rceil +(k-1)<d_G(v)/2+k$ which implies that
$|d_H(v)- \frac{1}{2}d_G(v)|<k$. For each $v\in X$, we therefore have $d_{H}(v)=d_G(v)/2$. This completes the proof.
}\end{proof}
For proving the remaining two theorems, we
need the following two lemmas.
Recall that a graph $G$ is said to be 
{\it $(m,l_0)$-partition-connected},
 if it can be decomposed into an $m$-tree-connected factor and a factor $F$ which admits an orientation such that for each vertex $v$, $d^+_F(v)\ge l_0(v)$,
where $l_0$ is a nonnegative integer-valued function on $V(G)$. 
\begin{lem}{\rm (\cite{ModuloBounded})}\label{Tree:lem:(k-1)-bound}
{Let $G$ be a graph, let $k$ be an integer, $k\ge 3$, and 
 let $p:V(G)\rightarrow \mathbb{Z}_k$ be a mapping with $|E(G)|\stackrel{k}{\equiv}\sum_{v\in V(G)}p(v)$.
Let $s$, $s_0$, and $l_0$ be three integer-valued functions on $V(G)$ satisfying $s(v)+s_0(v)+k-1\le d_G(v)$ and
$\max\{s(v),s_0(v)\}\le l_0(v)+(k-1)$ for each vertex $v$, and $\max\{s(z),s_0(z)\}\le l_0(z)$ for a vertex $z$. 
If $G$ is $(2k-2, l_0)$-partition-connected, then it has a $p$-orientation such that for each vertex $v$, 
$$s(v)\le d_{G}^+(v)\le d_G(v)-s_0(v).$$
}\end{lem}
\begin{lem}{\rm (\cite{ModuloBounded})}\label{lem:pre-orientation:tree-connected}
{Let $G$ be a graph of order at least two, let $k$ be an integer, $k\ge 3$, and let $p:V(G)\rightarrow \mathbb{Z}_k$ be a mapping with $|E(G)|\stackrel{k}{\equiv}\sum_{v\in V(G)}p(v)$. If $G$ is $(2k-2)$-tree-connected, then it has a $p$-orientation such that for each vertex~$v$, $$k/2-1\le d^+_G(v)\le d_G(v)-k/2+1.$$
}\end{lem}
The following theorem provides a partition-connected version for Theorem~\ref{thm:essentially:factor:k}.
\begin{thm}
{Let $G$ be a bipartite graph with bipartition $(X,Y)$, let $k$ be an integer, $k\ge 3$, and let $f:V(G)\rightarrow \mathbb{Z}_k$ be a mapping with $\sum_{v\in X}f(v) \stackrel{k}{\equiv}\sum_{v\in Y}f(v)$.
Let $s$, $s_0$, and $l_0$ be three integer-valued functions on $V(G)$ satisfying $s(v)+s_0(v)+k-1\le d_G(v)$ and
$\max\{s(v),s_0(v)\}\le l_0(v)+(k-1)$ for each vertex $v$, and $\max\{s(z),s_0(z)\}\le l_0(z)$ for a vertex $z$. 
If $G$ is $(2k-2, l_0)$-partition-connected, then it admits
an $f$-factor $H$ such that for each vertex $v$,
$$s(v) \le d_H(v) \le d_G(v)-s_0(v).$$
}\end{thm}
\begin{proof}
{
For each $v\in X$, define $ p(v) = f(v)$, $s'(v)=s(v)$, $s_0'(v)=s_0(v)$, and for each $v\in Y$, 
define $ p(v) = d_G(v)-f(v)$, $s'(v)=s_0(v)$, $s_0'(v)=s(v)$.
By the assumption, we must have $|E(G)|\stackrel{k}{\equiv}\sum_{v\in V(G)}p(v)$. 
By Lemma~\ref{Tree:lem:(k-1)-bound}, 
the graph $G$ has a $p$-orientation modulo $k$ such that for each vertex $v$,
$ s'(v) \le d^+_G(v) \le d_G(v)-s'_0(v)$.
Take $H$ to be the factor of $G$ consisting of all edges directed from $X$ to $Y$.
Since for all vertices $v\in X$, $d_H(v)=d_G^+(v)$, 
and for all vertices $v\in Y$, $d_H(v)=d_G(v)-d^+_G(v)$, the graph $H$ is the desired $f$-factor we are looking for.
}\end{proof}
An application of the following theorem is stated in Section~\ref{sec:modulo regular}.
\begin{thm}\label{thm:two-factors:modulo}
{Let $G$ be a bipartite graph of order at least two with bipartition $(X,Y)$, let $k$ be an integer, $k\ge 3$, and 
 let $f:V(G)\rightarrow \mathbb{Z}_k$ be a mapping with $\sum_{v\in X}f(v) \stackrel{k}{\equiv}\sum_{v\in Y}f(v)$.
If $G$ is $(2k-2)$-tree-connected, then it has an $f$-factor $H$ such that for each vertex $v$, 
$k/2-1 \le d_H(v)\le d_G(v)-k/2+1$.
}\end{thm}
\begin{proof}
{
For each $v\in X$, define $ p(v) = f(v)$, and for each $v\in Y$, define $ p(v) = d_{G}(v)-f(v)$.
By the assumption, we must have $|E(G)|\stackrel{k}{\equiv}\sum_{v\in V(G)}p(v)$. 
By Lemma~\ref{lem:pre-orientation:tree-connected}, the graph $G$ has a $p$-orientation
such that for each vertex $v$, $k/2-1\le d^+_G(v)\le d_G(v)-k/2+1$.
Define $H$ to be the factor of $G$ consisting of all edges directed from $X$ to $Y$.
It is easy to check that $H$ is the desired factor we are looking for. 
}\end{proof}
\subsection{Graphs with small bipartite index}
The following theorem establishes a non-bipartite version for Theorem~\ref{thm:essentially:factor:k} with a stronger version.
\begin{thm}\label{thm:bi:at-most:k-1}
{Let $G$ be a graph, let $k$ be an integer, $k\ge 3$, and let $f:V(G)\rightarrow \mathbb{Z}_k$ be a compatible mapping.
If $G$ is essentially $(3k-3)$-edge-connected and 
$d_G(v)\ge 2k-1+[f(v)]_k$ for each vertex $v$, and $e_G(X)+e_G(Y)\le k-1$ for a bipartition of $X,Y $ of $V(G)$,
then $G$ has an $f$-factor $H$ such that for each vertex~$v$,
$$\lfloor \frac{d_G(v)+s(v)}{2}\rfloor -(k-1)
\le d_H(v)\le 
 \lceil \frac{d_G(v)-s_0(v)}{2}\rceil +(k-1),$$
where $s$ and $s_0$ are two nonnegative integer-valued functions on $V(G)$ satisfying $s_0(v)+s(v)< 2k$ for each vertex $v$, and $e_G(X)+e_G(Y)+\frac{1}{2}\sum_{v\in V(G)}\max\{s(v),s_0(v)\}<k$.
}\end{thm}
\begin{proof}
{Let $M$ be the graph $G[X]\cup G[Y]$. Since $f$ is compatible, 
 there are two integers $x$ and $y$ with $0\le x\le e_G(X)=e_M(X)$ and $0\le y\le e_G(Y)=e_M(Y)$ such that 
$2y-2x \stackrel{k}{\equiv} \sum_{v\in Y}f(v)-\sum_{v\in X}f(v)$.
Let $M_1$ be a factor of $M$ such that $e_{M_1}(X)=x$ and $e_{M_1}(Y)=y$, and let $M_0=M\setminus E(M_1)$.
We are going to construct a new graph $L$ that plays an important role in the proof.
First, we add a new vertex $z_0$ to $G\setminus E(M)$. Next, for each edge $uv\in E(M)$,
 we add two edges $z_0u$ and $z_0v$ directed as follows.
Both them are directed toward $z_0$,
 if either $uv\in E(M_1)\cap E(G[X])$ or $uv\in E(M_0)\cap E(G[Y])$, 
and also directed away from $z_0$, 
 if either $uv\in E(M_0)\cap E(G[X])$ or $uv\in E(M_1)\cap E(G[Y])$.
 Note that we might have some multiple edges incident with $z_0$.
\begin{figure}[h]
 \centering
 \includegraphics[scale =2.2]{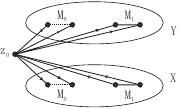}
 \caption{An orientation of all edges incident with $z_0$.}
 \label{L-z0}
\end{figure}
Call the resulting loopless graph $L$. Note also that $d_G(v)=d_L(v)$ for all $v\in V(G)$.
Define $p(z_0)=d^+_L(z_0)=2e_{M_0}(X)+2e_{M_1}(Y)$, and for each $v\in V(L)\setminus \{z_0\}$, define 
$$p(v)=
 \begin{cases}
d_{L}(v)-f(v), &\text{if $v\in Y$};\\
f(v), &\text{if $v\in X$}.
 \end{cases}$$
Thus
$$\sum_{v\in V(L)}p(v) 
\stackrel{k}{\equiv}
p(z_0)+\sum_{v\in V(G)}p(v)
\stackrel{k}{\equiv} 
2e_{M_0}(X)+2e_{M_1}(Y)+
\sum_{v\in X}f(v)+\sum_{v\in Y}(d_L(v)-f(v)),$$
and so
$$\sum_{v\in V(L)}p(v) 
\stackrel{k}{\equiv}2e_{M_1}(Y)-2e_{M_1}(X)+
\sum_{v\in X}f(v)- \sum_{v\in Y}f(v)+|E(L)|\stackrel{k}{\equiv}|E(L)|.$$
Obviously, 
$L$ is essentially $(3k-3)$-edge-connected, 
$d_{L}(v) = d_{G}(v)\ge 2k-1+[f(v)]_k=2k-1+[p(v)]_k$ for each $v\in X$, 
and $d_{L}(v) = d_{G}(v)\ge 2k-1+[f(v)]_k$ for each $v\in Y$
which can imply that $d_{L}(v) \ge 2k-1+[d_{L}(v) -f(v)]_k =2k-1+[p(v)]_k$.
In addition, $d_{L}(z_0)+\sum_{v\in V(G)}\max\{s'(v),s'_0(v)\}= 2|E(M)|+\sum_{v\in V(G)}\max\{s'(v),s'_0(v)\}< 2k$, 
where $s'(v)=s(v)$ and $s'_0(v)=s_0(v)$ when $v\in X$ and 
 $s'(v)=s_0(v)$ and $s'_0(v)=s(v)$ when $v\in Y$.
Therefore, 
by Theorem~\ref{thm:new-lower-bound}, the orientation of the edges of $L$ incident with $z_0$
 can be extended to a $p$-orientation of $L$ such that 
for each vertex $v$, 
$\lfloor (d_{L}(v)+s'(v))/2\rfloor-(k-1)\le d^+_{L}(v) \le \lceil (d_{L}(v)-s'_0(v)/2\rceil+(k-1).$
Let $F$ be the factor of $G$ consisting of all edges of $L-z_0$ directed from $X$ to $Y$.
Define $H= M_1\cup F$.
According to the construction of $H$, for each $v\in V(H)$, we have 
$$d_H(v)=d_{M_1}(v)+ d_{F}(v)=
 \begin{cases}
d^-_{L}(v), &\text{if $v\in Y$};\\
d^+_{L}(v), &\text{if $v\in X$}.
 \end{cases}$$
Thus
$\lfloor (d_L(v)+s'(v))/2\rfloor-(k-1)
\le d_H(v)\le 
 \lceil (d_L(v)-s'_0(v))/2\rceil+(k-1)$.
Hence it is not hard to check that $H$ is the desired $f$-factor we are looking for.
}\end{proof}
The following corollary plays an essential role in the proof of Theorem~\ref{thm:high-enough-tree-connectivity}
in Subsection~\ref{subsec:Improving-degree-bounds}.
\begin{cor}\label{cor:bi:at-most:k-1}
{Let $G$ be a graph, let $k$ be an integer, $k\ge 3$, and let $f:V(G)\rightarrow \mathbb{Z}_k$ be a compatible mapping.
If $G$ is $(3k-3)$-edge-connected and $e_G(X)+e_G(Y)\le k-1$ for a bipartition of $X,Y $ of $V(G)$,
then $G$ has an $f$-factor $H$ such that for each vertex~$v$,
$$\lfloor \frac{d_G(v)}{2}\rfloor -(k-1)
\le d_H(v)\le 
 \lceil \frac{d_G(v)}{2}\rceil +(k-1).$$
Furthermore, for an arbitrary given vertex $z$ of odd degree, the upper bound and the lower bound can be reduced by one.
}\end{cor}
\begin{proof}
{Apply Theorem~\ref{thm:bi:at-most:k-1} by setting $s(z)=s_0(z)=1$ and $s(v)=s_0(v)=0$ for all $v\in V(G)\setminus \{z\}$.
}\end{proof}
\label{subsec:Non-bipartite general graphs}
The following theorem establishes a non-bipartite version for Corollary~\ref{cor:bipartite:factor} and plays an essential role in the subsequent section.
Let $Q$ be a trail-decomposition of the edges of a graph $G$ and let $X\subseteq V(G)$. 
Here, we say that $Q$ is 
{\it $X$-parity trail-decomposition}, 
if every trail in $Q$ of odd size has exactly one end in $X$ 
and every trail in $Q$ of even size has both ends in either $X$ or $V(G)\setminus X$.
\begin{thm}\label{thm:non-bipartite:trail}
{Let $G$ be a graph, let $k$ be an integer, $k\ge 3$, and let $f:V(G)\rightarrow \mathbb{Z}_k$ be a compatible 
mapping. Let $G_0$ be a factor of $G$ such that its complement admits an $X$-parity trail-decomposition and
$e_{G_0}(X)+e_{G_0}(Y)=k-1$, where $X,Y $ is a bipartition of $V(G)$.
If $G_0$ is $(3k-3)$-edge-connected,
then $G$ has an $f$-factor $H$ such that for each vertex~$v$,
$$\lfloor \frac{d_G(v)}{2}\rfloor-(k-1) \le d_H(v) \le \lceil \frac{d_G(v)}{2}\rceil+(k-1).$$
Furthermore, for an arbitrary given vertex $z$ of odd degree, the upper bound and the lower bound can be reduced by one.
}\end{thm}
\begin{proof}
{Let $Q$ be an $X$-parity trail-decomposition of $G\setminus E(G_0)$.
Let $T$ be a graph with $V(T)=V(G)$ and for each trail in $Q$ having different end vertices $v$ and $u$ add the edge $uv$ in $T$; adding parallel edges if necessary. Note that $T$ is loopless.
For each vertex $v$, we denote by $t(v)$ the number of times trails pass through $v$ but not finish and start at $v$ plus the number of closed trails started at $v$.
Let $M=G_0[X]\cup G_0[Y]$ so that $|E(M)|= k-1$.
Since $f$ is compatible, $\sum_{v\in X}f(v)-\sum_{v\in Y}f(v)$ is even when $k$ is even.
On the other hand, since $Q$ is an $X$-parity trail-decomposition, one can conclude that $\sum_{v\in Y}t(v)-\sum_{v\in X}t(v)+e_T(Y)-e_T(X)$ must be even (by splitting into similar expressions corresponding to trails of $Q$).
Thus there is a factor $M_1$ of $M$ such that 
 $2e_{M_1}(Y)-2e_{M_1}(X) \stackrel{k}{\equiv} \sum_{v\in Y}(f(v)-t(v))- \sum_{v\in X}(f(v)-t(v))-e_T(Y)+e_T(X)$.
Let $M_0=M\setminus E(M_1)$.

We are going to construct a new graph $L$ similarly to the proof of Theorem~\ref{thm:bi:at-most:k-1}.
First, we add a new vertex $z_0$ to $G_0\setminus E(M)$. Next, for each edge $uv\in E(M)$,
 we add two edges $z_0u$ and $z_0v$ directed as follows.
Both them are directed toward $z_0$,
 if either $uv\in E(M_1)\cap E(G[X])$ or $uv\in E(M_0)\cap E(G[Y])$, 
and also directed away from $z_0$, 
 if either $uv\in E(M_0)\cap E(G[X])$ or $uv\in E(M_1)\cap E(G[Y])$.
Call the resulting loopless graph $L$.
 Note that for each $v\in V(G)$, 
$d_G(v)=d_{L}(v)+d_{T}(v)+2t(v)$.
Define $p(z_0)=d^+_L(z_0)=2e_{M_0}(X)+2e_{M_1}(Y)$, and 
for each $v\in V(L)\setminus \{z_0\}$, define 
$$p(v)=
 \begin{cases}
d_{L}(v)+d_{T}(v)-(f(v)-t(v)), &\text{if $v\in Y$};\\
f(v)-t(v), &\text{if $v\in X$}.
 \end{cases}$$
Thus
$$\sum_{v\in V(L)}p(v) 
\stackrel{k}{\equiv}
p(z_0)+\sum_{v\in V(G)}p(v)
\stackrel{k}{\equiv} 
2e_{M_0}(X)+2e_{M_1}(Y)+
\sum_{v\in X}(f(v)-t(v))+\sum_{v\in Y}(d_{L}(v)+d_{T}(v)-(f(v)-t(v))),$$
and so
$$\sum_{v\in V(L)}p(v) 
\stackrel{k}{\equiv}2e_{M_1}(Y)-2e_{M_1}(X)+e_T(Y)-e_T(X)+
\sum_{v\in X}(f(v)-t(v))- \sum_{v\in Y}(f(v)-t(v))+|E(L)|+|E(T)|.$$
This implies that $\sum_{v\in V(L)}p(v)\stackrel{k}{\equiv}|E(L)|+|E(T)|$.
Obviously, $L$ is essentially $(3k-3)$-edge-connected, 
$d_{L}(v) = d_{G_0}(v)\ge 3k-3 \ge 2k-1+[p(v)]_k$ for all $v\in V(L)\setminus \{z_0\}$, 
and $d_{L}(z_0)+d_{T}(z_0)+1\le 2|E(M)|+1= 2k-1$. 
Thus
by Theorem~\ref{thm:new-lower-bound}, the orientation of the edges of $L$ incident with $z_0$
 can be extended to a $p$-orientation of $L\cup T$ such that 
for each vertex $v$, 
$\lfloor \frac{1}{2}(d_{L}(v)+d_{T}(v))\rfloor-(k-1) \le 
d^+_{T}(v)+d^+_{L}(v) \le 
\lceil \frac{1}{2}(d_{L}(v)+d_{T}(v))\rceil+(k-1).$
In addition, we can have 
$\lfloor \frac{1}{2}(d_{L}(z)+d_{T}(z)+1)\rfloor-(k-1) \le 
d^+_{T}(z)+d^+_{L}(z) \le 
\lceil \frac{1}{2}(d_{L}(z)+d_{T}(z)-1)\rceil+(k-1).$

Define $F$ to be the factor of $G_0$ consisting of all edges of $L-z_0$ directed from $X$ to $Y$.
Let $v_0,\ldots, v_n$ be an arbitrary trail in $Q$ such that the edge $v_0v_n$ directed from $v_0$ to $v_n$ in $T$.
If $v_0\in X$, we select all edges $v_{2i}v_{2i+1}$ of this trail, and if $v_0\in Y$, we select all edges $v_{2i+1}v_{2i+2}$.
Let $T'$ be the factor of $T$ consists of all selected edges.
Since $Q$ is an $X$-parity trail-decomposition, we must have
 $d_{T'}(v)=d^+_{T}(v)+t(v)$ for each $v\in X$,
and $d_{T'}(v)=d^-_{T}(v)+t(v)$ for each $v\in Y$.
Define $H= M_1\cup F\cup T'$.
According to the construction of $H$, for each $v\in V(H)$, we have 
$$d_H(v)=d_{M_1}(v)+ d_{F}(v)+d_{T'}(v)=
 \begin{cases}
d^-_{L}(v)+d^-_{T}(v)+t(v), &\text{if $v\in Y$};\\
d^+_{L}(v)+d^+_{T}(v)+t(v), &\text{if $v\in X$}.
 \end{cases}$$
Therefore, 
$\lfloor d_G(v)/2\rfloor-(k-1)
\le d_H(v)\le 
 \lceil d_G(v)/2\rceil+(k-1)$,
and also $\lfloor (d_G(z)+1)/2\rfloor-(k-1)
\le d_H(z)\le 
 \lceil (d_G(z)-1)/2\rceil+(k-1)$.
Hence it is not hard to check that $H$ is the desired $f$-factor we are looking for.
}\end{proof}
%
%
%
%
%
%
%
%
%
%
%
%
\section{Factors modulo $k$: General graphs}
Our aim in this section is to generalize Corollary~\ref{cor:bipartite:factor} 
to general graphs using the same degree bounds and characterize the exceptional graphs with high enough edge-connectivity. We begin with a similar version by increasing the upper bound $\lceil d_G(v)/2\rceil+k-1$ to 
 $\lfloor d_G(v)/2\rfloor+k$.
\subsection{General graphs}
\label{subsec:General graphs}
The following theorem improves the condition of edge-connectivity of Theorem~\ref{intro:thm:non-bipartite:6k-7}.
\begin{thm}\label{thm:3k-3:non-bipartite}
{Let $G$ be a graph, let $k$ be a positive integer, and let $f:V(G)\rightarrow \mathbb{Z}_k$ be a compatible 
mapping.
If $G$ contains a $(3k-3)$-edge-connected bipartite factor,
then $G$ has an $f$-factor $H$ such that for each vertex~$v$,
$$\lfloor \frac{d_G(v)}{2}\rfloor-(k-1) \le d_H(v) \le \lfloor \frac{d_G(v)}{2}\rfloor+k.$$
}\end{thm}
\begin{proof}
{The special case $k=1$ can be proved by Corollary~\ref{cor:Eulerian:1/2}.
More precisely, 
when $G$ is not Eulerian, 
we need to add an artificial vertex $z$ and join it to all vertices with odd degrees.
Moreover, Theorem~\ref{Factor:modulo2:thm:2:edge} confirms the case $k=2$.
So, suppose $k\ge 3$. Let $X,Y$ be a bipartition of $V(G)$ such that $G[X,Y]$ is $(3k-3)$-edge-connected.
 Define $G_0$ to be a factor of $G$ containing all edges of $G[X,Y]$ such that 
$e_{G_0}(X)+e_{G_0}(Y)=\min\{k-1, e_G(X)+e_G(Y)\}$. 
If $G_0=G$, then the assertion follows from Corollary~\ref{cor:bi:at-most:k-1}.
Thus we may assume that $e_{G_0}(X)+e_{G_0}(Y)\ge k-1$.
Let $T$ be a factor of $G$ obtained by selecting edge-disjoint trails of length two from $G\setminus E(G_0)$ as long as possible. Since both ends of each selected trails lie either in $X$ or in $Y$, the graph $T$ admits an $X$-parity trail-decomposition. Let $M=G\setminus (E(G_0)\cup E(T))$. According to the construction of $T$, the graph $M$ must be a matching. For each vertex $v$, define $f'(v)=f(v)-d_M(v)$ (mod $k$).
Since $f$ is compatible with $G$, $\sum_{v\in V(G)}f'(v)$ must be even when $k$ is even.
Therefore, $f'$ must be compatible with $G_0\cup T$ by applying Theorem~\ref{thm:compatible:sufficient} (iii).
Thus by Theorem~\ref{thm:non-bipartite:trail}, the graph $G_0\cup T$ has an $f'$-factor $F$ such that for each vertex $v$, 
$\lfloor (d_G(v)-d_M(v))/2 \rfloor -(k-1)
\le d_F(v)\le 
 \lceil (d_G(v)-d_M(v))/2\rceil +(k-1).$
Hence it is not hard to check that $F\cup M$ is the desired $f$-factor we are looking for.
}\end{proof}
\begin{cor}\label{cor:non-bipartite:6k-7}
{Let $G$ be a graph, let $k$ be a positive integer, and let $f:V(G)\rightarrow \mathbb{Z}_k$ be a mapping.
 Assume that $(k-1)\sum_{v\in V(G)}f(v)$ is even and $bi(G)\ge k_0-1$, where $k_0=k$ when $k$ is odd, and $k_0=k/2$ when $k$ is even. If $G$ is $(6k-7)$-edge-connected, then it has an $f$-factor $H$ such that for each vertex~$v$,
$$\lfloor \frac{d_G(v)}{2}\rfloor-(k-1) \le d_H(v) \le \lfloor \frac{d_G(v)}{2}\rfloor+k.$$
}\end{cor}
\begin{proof}
{By Theorem~\ref{thm:bipartite:factor:edge-version}, the graph $G$ has a bipartite $(3k-3)$-edge-connected factor.
Also, by Theorem~\ref{thm:compatible:sufficient}, the mapping $f$ is compatible with $G$.
Thus the assertion follows from Theorem~\ref{thm:3k-3:non-bipartite}.
}\end{proof}
The following corollary is a supplement of Corollary~\ref{cor:Eulerian:even-k} for the special case $k=2$.
\begin{cor}
{Every non-bipartite $18$-edge-connected Eulerian graph $G$ of even size admits a factor $H$ 
such that for each vertex $v$, $$d_H(v)\in \{d_G(v)/2-2,d_G(v)/2+2\}.$$
}\end{cor}
\begin{proof}
{For each $v$, define $f(v)=d_G(v)/2+2$ (mod $4$).
Since $G$ is non-bipartite and $|E(G)|$ is even,
 $bi(G)\ge 1=4/2-1$ and $\sum_{v\in V(G)}f(v)\stackrel{2}{\equiv}|E(G)|\stackrel{2}{\equiv}0$.
 This implies that $f$ is compatible with $G$ (modulo $4$) using Theorem~\ref{thm:compatible:sufficient}.
Thus by applying Corollary~\ref{cor:non-bipartite:6k-7} with $k=4$,
the graph $G$ has an $f$-factor $H$ such that for each vertex~$v$, $d_G(v)/2-3 \le d_H(v) \le d_G(v)/2+4$.
Since $d_H(v)\stackrel{4}{\equiv}d_G(v)/2+2$, we must have $d_H(v)\in \{d_G(v)/2-2,d_G(v)/2+2\}$.
Hence the proof is completed.
}\end{proof}
The following result is an interesting application of Theorem~\ref{thm:3k-3:non-bipartite}.
This result will be refined for $6k$-tree-connected graphs
by replacing Theorem~\ref{thm:high-enough-tree-connectivity} in the proof.
\begin{cor}\label{cor:k-1/2}
{Let $G$ be a graph, let $k$ be a positive integer, and let $f$ be a positive integer-valued function on $V(G)$ satisfying
 $f(v) \le \frac{1}{2}d_G(v)< f(v)+k$ for each vertex $v$. 
Assume that $f$ is compatible with $G$ (modulo $k$).
If $G$ is $(6k-7)$-edge-connected, then $G$ has a factor $H$ such that for each vertex $v$,
$$d_H(v)\in \{f(v),f(v)+k\}.$$
}\end{cor}
\begin{proof}
{By Theorem~\ref{thm:bipartite:factor:edge-version}, the graph $G$ has a bipartite $(3k-3)$-edge-connected factor. Thus by Theorem~\ref{thm:3k-3:non-bipartite}, the graph $G$ has an $f$-factor $H$ (mod $k$) such that for each vertex $v$, $f(v) -(k-1)\le \lfloor d_G(v)/2\rfloor -(k-1)\le d_H(v) \le \lfloor d_G(v)/2\rfloor+k\le f(v)+2k-1$.
Hence $H$ is the desired factor.
}\end{proof}
We can also formulate the following theorem similar to Theorem~\ref{thm:3k-3:non-bipartite}.
\begin{thm}
{Let $G$ be a graph of order at least two, let $k$ be an integer, $k\ge 3$, and let $f:V(G)\rightarrow \mathbb{Z}_k$ be a compatible mapping. 
If $G$ contains a $(2k-2)$-tree-connected bipartite factor, then it has an $f$-factor $H$ such that for each vertex $v$, 
$k/2-1\le d_H(v)\le d_G(v)-k/2+1$.
}\end{thm}
\begin{proof}
{Let $X,Y$ be a bipartition of $V(G)$ such that $G[X,Y]$ is $(2k-2)$-tree-connected.
Let $G_0=G[X,Y]$.
Since $f$ is compatible, there are two nonnegative integers $x$ and $y$ with $x\le e_G(X)$ and 
$y\le e_G(Y)$ and
 $\sum_{v\in X}f(v)-2x \stackrel{k}{\equiv} \sum_{v\in Y}f(v)-2y$.
Let $M$ be a factor of $G[X]\cup G[Y]$ such that $e_M(X)=x$ and $e_M(Y)=y$.
 For each vertex $v$, define $f'(v)=f(v)-d_M(v)$. 
Obviously, $\sum_{v\in X}f'(v) \stackrel{k}{\equiv}\sum_{v\in Y}f'(v)$.
Thus by Theorem~\ref{thm:two-factors:modulo}, the graph $G_0$ has an $f'$-factor $F$ such that for each vertex $v$, 
$k/2-1\le d_F(v)\le d_{G_0}(v)-k/2+1$.
Hence it is not hard to check that $F\cup M$ is the desired $f$-factor we are looking for.
}\end{proof}
%
%
%
%
%
%
%
%
\subsection{Improving degree bounds: highly edge-connected graphs}
\label{subsec:Improving-degree-bounds}
The following result improves the upper bound stated in Theorem~\ref{thm:3k-3:non-bipartite} for highly edge-connected graphs.
\begin{thm}\label{thm:high-enough-tree-connectivity}
{Let $G$ be a graph with $z\in V(G)$, let $k$ be a positive integer, and let $f:V(G)\rightarrow \mathbb{Z}_k$ be a compatible
mapping. 
If $G$ is $(6k-2)$-tree-connected, 
then $G$ has an $f$-factor $H$ such that for each vertex~$v$,
$$\lfloor \frac{d_G(v)}{2}\rfloor-(k-1)\le d_H(v) \le 
 \begin{cases}
\lfloor \frac{d_G(v)}{2}\rfloor+k, &\text{when $v= z$};\\
\lceil \frac{d_G(v)}{2}\rceil+(k-1), &\text{otherwise}.
 \end{cases}$$
Furthermore, for the vertex $z$, the upper bound can be reduced to $\lceil d_G(z)/2\rceil+(k-1)$ if and only if one of the following conditions hold:
\begin{enumerate}{
\item [$\bullet$] $k$ is even.
\item [$\bullet$] $G$ has a vertex of odd degree.
\item [$\bullet$] $G$ is an Eulerian graph of even size.
\item [$\bullet$] There is a vertex $v$ for which $f(v)\not \stackrel{k}{\equiv}d_G(v)/2$.
}\end{enumerate}
}\end{thm}
\begin{proof}
{The special case $k=1$ can be proved by Corollary~\ref{cor:Eulerian:1/2}; see \cite[Problem 41, Page 61]{Lovasz-1979}. More precisely, 
when $G$ is not Eulerian, 
we need to add an artificial vertex $z'$ and join it to all vertices with odd degrees, and this new vertex should play the role of the vertex $z$ in Corollary~\ref{cor:Eulerian:1/2}. 
Moreover, Theorem~\ref{Factor:modulo2:thm:2:edge} confirms the case $k=2$.
So, suppose $k\ge 3$.
If $bi(G)\le k-1$, then the statement follows from 
Corollary~\ref{cor:bi:at-most:k-1} and Theorem~\ref{thm:bipartite:factor:edge-version}.
We may assume that $bi(G)\ge k$.
By Corollary~\ref{cor:decomposition:bipartite-index}, the graph $G$ can be decomposed into two factors $G_1$ and $G_2$ such that $G_1$ is Eulerian and $G_2$ is $(3k-3)$-edge-connected, and $e_{G_2}(X)+e_{G_2}(Y) =k-1$ for a bipartition $X,Y$ of $V(G)$.
According to Corollary~\ref{cor:Eulerian:1/2}, the graph $G_1$ has a factor $F_1$ such that 
for each $v\in V(G)\setminus \{z\}$, $d_{F_1}(v)=d_{G_1}(v)/2$, and $d_{F_1}(z)\in \{d_{G_1}(z)/2, d_{G_1}(z)/2+1\}$. 
For each vertex $v$, define $f'(v)=f(v)-d_{F_1}(v)$ (mod $k$).
Since $f$ is compatible with $G$, $(k-1)\sum_{v\in V(G)}f(v)$ is even and hence $(k-1)\sum_{v\in V(G)}f'(v)$ is even.
Thus by Theorem~\ref{thm:compatible:sufficient}, $f'$ must be compatible with $G_2$.
According to Corollary~\ref{cor:bi:at-most:k-1}, 
the graph $G_2$ has an $f'$-factor $F_2$ such that for each vertex $v$, 
 $\lfloor d_{G_2}(v)/2\rfloor-(k-1) \le d_{F_2}(v) \le \lceil d_{G_2}(v)/2\rceil+(k-1)$, and also
$d_{F_2}(z) \le \lfloor d_{G_2}(z)/2\rfloor+(k-1) $.
Hence $H=F_1\cup F_2$ is the desired factor we are looking for.

Assume that the upper bound cannot be reduced to $\lceil d_G(z)/2\rceil+(k-1)$.
If $d_G(z)$ is odd, then obviously the upper bound is the same number $\lceil d_G(z)/2\rceil+(k-1)$.
Thus we may assume that all vertices of $G$ have even degree (otherwise, we can select a vertex of odd degree for playing the role of $z$). If $f(z)\not \stackrel{k}{\equiv} d_G(z)/2$, then obviously the upper bound can also be reduced to $\lceil d_G(z)/2\rceil+(k-1)$. Thus we may assume that $f(v)\stackrel{k}{\equiv}d_G(v)/2$ for all vertices $v$. 
On the other hand, by applying Corollary~\ref{cor:Eulerian:1/2}, the graph $G$ has a factor $H$ such that 
for each $v\in V(G)\setminus \{z\}$, $d_{H}(v)=d_{G}(v)/2$, and $d_{H}(z)\in \{d_{G}(z)/2, d_{G}(z)/2+1\}$. 
If $|E(G)|$ is even, then $\sum_{v\in V(G)}d_{G}(v)/2$ is even which implies that $d_{H}(z)=d_{G}(z)/2$. Thus $|E(G)|$ must be odd.
If $k$ is even, then $\sum_{v\in V(G)}f(v)$ is even and hence $|E(G)|$ is even, which is a contradiction.
 Therefore, $k$ must be odd.

Conversely, assume that $G$ is an Eulerian graph of odd size and $G$ has an $f$-factor $H$ such that for each vertex $v$, $f(v)\stackrel{k}{\equiv}d_G(v)/2$ and
$\lfloor d_G(v)/2\rfloor -(k-1)\le d_H(v)\le \lceil d_G(v)/2\rceil+k-1$.
 This implies that $d_H(v)=d_G(v)/2$ for each vertex $v$.
Since $\sum_{v\in V(G)}d_H(v)$ is even, $|E(G)|$ must be even, which is a contradiction. Hence the proof is completed.
}\end{proof}
\begin{remark}
{Note that the needed tree-connectivity of Theorem~\ref{thm:high-enough-tree-connectivity} 
can be reduced by one and two for odd and even integers $k$ using a little extra effort.
}\end{remark}
The following corollary is a refined version of Corollary~\ref{cor:k-1/2} for highly tree-connected graphs.
\begin{cor}
{Let $G$ be a graph, let $k$ be a positive integer, and let $f$ be a positive integer-valued function on $V(G)$ satisfying
 $f(v) \le \frac{1}{2}d_G(v)\le f(v)+k$ for each vertex $v$. 
Assume that $f$ is compatible with $G$ (modulo $k$).
If $G$ is $(6k-2)$-tree-connected, then $G$ has a factor $H$ such that for each vertex $v$,
$$d_H(v)\in \{f(v),f(v)+k\}.$$
}\end{cor}
\begin{proof}
{Let $z\in V(G)$. By Theorem~\ref{thm:high-enough-tree-connectivity},
the graph $G$ has an $f$-factor $H$ (mod $k$) such that for each $v\in V(G)\setminus \{z\}$, 
$f(v) -(k-1)\le \lfloor d_G(v)/2\rfloor -(k-1)\le d_H(v) \le \lceil d_G(v)/2\rceil+(k-1)\le f(v)+2k-1$ which implies that 
$d_H(v)\in \{f(v),f(v)+k\}$.
In addition, we can have $f(z) -(k-1)\le \lfloor d_G(z)/2\rfloor -(k-1)\le d_H(z) \le \lceil d_G(z)/2\rceil+k$.
If the last upper bound can be improved by one or $d_G(z)/2\le f(z)+k-1$, then we must also
 have $d_H(z)\in \{f(z),f(z)+k\}$.
Otherwise, $G$ must be Eulerian and $d_G(v)/2=f(v)+k$ for all vertices $v$.
In this case, we first apply Theorem~\ref{thm:high-enough-tree-connectivity} to find a factor $H_0$ such that for each vertex $v$,
$d_{H_0}(v)\in \{d_G(v)-f(v)-k, d_G(v)-f(v)\}$. Next, we set $H=G\setminus E(H_0)$.
Note that the function $d_G(v)-f(v)$ is compatible with $G$ as well. 
More precisely, if there are two integers $x$ and $y$ satisfying $0\le x\le e_G(X)$, 
$0\le y\le e_G(Y)$, and
 $\sum_{v\in X}f(v)-2x \stackrel{k}{\equiv} \sum_{v\in Y}f(v)-2y$, then 
 $\sum_{v\in X}(d_G(v)-f(v))-2(e_G(X)-x) \stackrel{k}{\equiv} \sum_{v\in Y}(d_G(v)-f(v))-2(e_G(Y)-y)$, where $X, Y$ is an arbitrary bipartition of $V(G)$.
Hence the proof is completed.
}\end{proof}
%
%
%
%
%
%
%
%
%
%
%
%
%
%
\section{Modulo $k$-regular factors and subgraphs}
\label{sec:modulo regular}
\subsection{Bipartite modulo $k$-regular factors}

The following well-known theorem gives a sufficient condition for the existence of even factors.
In this subsection, we develop this result for the existence of bipartite modulo $k$-regular factors.
\begin{thm}\rm{(Lov{\'a}sz~\cite[Problem 42, Page 61]{Lovasz-1979}})\label{thm:evenfactor:Eulerian}
{Every $2$-edge-connected loopless graph $G$ with $\delta(G)\ge 3$ admits a modulo $2$-regular factor.
}\end{thm} 
We begin with the following corollary which provides a bipartite version for Theorem~\ref{thm:evenfactor:Eulerian}.
Note that this result is sharp by considering that there exits a class of $4$-edge-connected graphs without bipartite modulo $2$-regular factors. (For example, consider a number of copies of the complete graph of order four and add a new vertex and join it to all other vertices). 
\begin{cor}
{Every $3$-edge-connected loopless graph $G$ with $\delta(G)\ge 5$ admits 
a bipartite modulo $2$-regular factor.
}\end{cor} 
\begin{proof}
{By Theorem~\ref{thm:bipartite:factor}, the graph $G$ has a bipartite factor $H$ such that for every vertex set $X$,
 $d_H(X)\ge \lceil d_G(X)/2 \rceil$.
Since $G$ is $3$-edge-connected and $\delta(G)\ge 5$, the graph $H$ must $2$-edge-connected and $\delta(H)\ge 3$. Hence by Theorem~\ref{thm:evenfactor:Eulerian} the graph $H$ admits a modulo $2$-regular factor.
}\end{proof}
The following theorem gives sufficient edge-connectivity conditions for the existence of bipartite modulo $k$-regular factors. 
\begin{thm}\label{thm:regular:essentially}
{Every $(4k-1)$-edge-connected essentially $(6k-7)$-edge-connected graph with $k\ge 3$ admits 
a bipartite modulo $k$-regular factor. In addition, this result is true 
for bipartite $2k$-edge-connected essentially $(3k-3)$-edge-connected graphs.
}\end{thm}
\begin{proof}
{By Theorem~\ref{thm:bipartite:factor}, the graph $G$ has a bipartite factor $G_0$ such that for every vertex set $X$,
 $d_{G_0}(X)\ge \lceil d_G(X)/2 \rceil $. Since $G$ is $(4k-1)$-edge-connected and essentially $(6k-7)$-edge-connected, the graph $G_0$ must be $2k$-edge-connected and essentially $(3k-3)$-edge-connected.
For each vertex $v$, define $f(v)=0$. Obviously, $f$ is compatible with $G$.
Thus by Theorem~\ref{thm:essentially:factor:k}, 
the graph $G_0$ has an $f$-factor $F$ modulo $k$
such that for each vertex $v$, $d_F(v) \ge \lfloor d_{G_0}(v)/2\rfloor-(k-1)> 0$.
Thus $F$ is a bipartite modulo $k$-regular factor of $G$.
}\end{proof}
The following theorem provides a tree-connected version for Theorem~\ref{thm:regular:essentially}.
\begin{thm}
{Every $(4k-4)$-tree-connected graph with $k\ge 3$ admits a bipartite modulo $k$-regular factor.
In addition, this result is true for bipartite $(2k-2)$-tree-connected graphs.
}\end{thm}
\begin{proof}
{If $G$ is $(4k-4)$-tree-connected, then by Theorem~\ref{thm:bipartite:factor}, the graph $G$ contains
a $(2k-2)$-tree-connected bipartite factor $G_0$. By Theorem~\ref{thm:two-factors:modulo},
the bipartite graph $G_0$ contains
a modulo $k$-regular factor, which can complete the proof.
}\end{proof}
Finally, we formulate the following improved version of Theorem 3 in~\cite{Thomassen-2014}.
\begin{thm}\label{thm:moduloregular:edgeversion}
{Every $(10k-3)$-edge-connected essentially $(12k-7)$-edge-connected graph of even order admits a modulo $k$-regular factor whose degrees are not divisible by $2k$. In addition, this result is true for bipartite $(5k-1)$-edge-connected essentially $(6k-3)$-edge-connected graphs of even order.
}\end{thm}
\begin{proof}
{By Theorem~\ref{thm:bipartite:factor}, the graph $G$ has a bipartite factor $G_0$ such that for every vertex set $X$,
 $d_{G_0}(X)\ge \lceil d_G(X)/2 \rceil $. Since $G$ is $(10k-3)$-edge-connected and essentially $(12k-7)$-edge-connected, the graph $G_0$ must be $(5k-1)$-edge-connected and essentially $(6k-3)$-edge-connected.
For each vertex $v$, define $f(v)=k$ (mod $2k$). Since $G_0$ has even order, $f$ must be compatible with $G_0$.
Thus by Theorem~\ref{thm:essentially:factor:k}, 
the graph $G_0$ has an $f$-factor $F$ modulo $2k$. Note that for each vertex $v$, we have $d_{G_0}(v)\ge 5k-1= 2(2k)-1+[f(v)]_{2k}$. Thus $F$ is the desired factor of $G$.
}\end{proof}
\subsection{Bipartite modulo $k$-regular subgraphs}
The following theorem completely confirms Conjecture~\ref{intro:conj} for prime powers.
We shall below replace another condition for the existence of bipartite modulo $k$-regular subgraphs.
\begin{thm}{\rm (\cite{Alon-Friedland-Kalai-1984})}\label{thm:AFK:k-subgraph}
{Let $G$ be a loopless graph of order $n$ and let $q$ be a prime power.
Then $G$ admits a modulo $q$-regular subgraph if 
$$|E(G)|> (q-1)n.$$
In addition, the lower bound can be improved to $(q-1)(n-1)$ when $G$ is bipartite.
}\end{thm} 
\begin{lem}{\rm (\cite{Hoffman})}\label{lem:kf-factor}
{Let $k$ and $q$ be two positive integers with $k\le q$. 
Let $G$ be a bipartite graph and let $f$ be a nonnegative integer-valued function on $V(G)$.
If $G$ has a factor $H$ satisfying $d_H(v)=qf(v)$ for each vertex $v$, then $G$ has a factor $F$ satisfying $d_F(v)=kf(v)$ for each vertex $v$.
}\end{lem} 
\begin{proof}
{We split every vertex $v$ of $H$ into $f(v)$ vertices such that the resulting graph $H_0$ would be a $q$-regular bipartite graph. Thus $H_0$ has a $k$-regular factor $F_0$ using K{\"o}nig's Theorem~\cite{Konig}. Obviously, this factor $F_0$ induces a factor $F$ of $H$ satisfying $d_F(v)=kf(v)$ for each vertex $v$. Hence the proof is completed.
}\end{proof} 
The following theorem is an improved version of Lemma 3 in \cite{Botler-Colucci-Kohayakawa}.
By a computer search, we observed that for small positive integers $k$ (at most nine digits), there are some prime powers less than $(1+1/10)k+1$. On the other hand, Baker, Harman, and Pintz (2001)~\cite{Baker-Harman-Pintz-2001} proved that for every sufficiently large integer $x$, there is a prime number $p\in [x-x^{0.525}, x]$. This shows that the following lower bound can be replaced by $(2+\varepsilon)(k -1)(n-1)$
for sufficiently large $k$.
\begin{thm}\label{thm:qk}
{A loopless graph $G$ of order $n$ has a bipartite modulo $k$-regular subgraph if $\chi(G)\le 2t$ and 
$$|E(G)|> (2-\frac{1}{t})(q(k)-1)(n-1),$$
where $q(k)$ denotes the smallest prime power with $ q(k) \ge k$. Consequently, 
if $|E(G)|> (2q(k)-\frac{5}{2})(n-1)$ or $|E(G)|> (4k-6)(n-1)$, then $G$ has a bipartite modulo $k$-regular subgraph.
}\end{thm}
\begin{proof}
{We may assume that $k\ge 2$. We first show that the graph $G$ has a bipartite factor $H$ such that $|E(H)|\ge \frac{t}{2t-1}|E(G)|$, see \cite{Andersen-Grant-Linial-1983}. 
Let $X_1,\ldots, X_{2t}$ be a partition of $V(G)$ such that every $G[X_i]$ has no edge.
Let $S$ be a subset of $1,\ldots, 2t$ with size $t$, and 
let $H_S$ be the bipartite factor of $G$ with one partite set $\cup_{i\in S} X_i$.
The number of such factors is obviously $\binom{2t}{t}$. On the other hand, every edge is contained in exactly
$2\binom{2t-2}{t-1}$ such factors. Among all such factors, we consider $H$ with the maximum size.
Thus $\binom{2t}{t} |E(H)|\ge 2\binom{2t-2}{t-1}|E(G)|$ which implies that
$|E(H)|\ge \frac{t}{2t-1}|E(G)|> (q(k)-1)(n-1)$.
Thus by Theorem~\ref{thm:AFK:k-subgraph}, the graph $H$ has a modulo $q(k)$-regular subgraph $F$.
According to Lemma~\ref{lem:kf-factor}, the bipartite graph $F$ has a modulo $k$-regular factor and so the graph $G$ has a bipartite modulo $k$-regular subgraph.
 
Let us prove the first conclusion. Suppose, to the contrary, $G$ has no bipartite modulo $k$-regular subgraph and $|E(G)|> (2q(k)-5/2)(n-1)$. Thus every subgraph $G'$ of $G$ contains at most $2(q(k)-1)|V(G')|-1$ edges, and so it has minimum degree at most $4q(k)-5$. Therefore, one can conclude that the graph $G$ its chromatic number is at most $4q(k)-4$ (using a greedy algorithm with the fact that $G$ is $(4q(k)-5)$-degenerate). Thus the graph $G$ contains at most $(2-\frac{1}{2q(k)-2})(q(k)-1)(n-1)$ edges, which is a contradiction. For proving the second conclusion, it is enough to check that $q(k) \le 2^{i+1}\le 2k-2$ 
when $k=2^i+r$ and $0< r<2^i$. Hence the proof is completed.
}\end{proof} 
 When $k=3$ and $t=2$, Theorem~\ref{thm:qk} becomes simpler as the following version. Note that the graph $G$ obtained from the Cartesian product of two cycles of order three by joining its vertices to a new vertex is a $4$-chromatic simple graph of size $3(|V(G)|-1)$ having no bipartite modulo $3$-regular subgraph.
\begin{cor}\label{cor:k=3}
{A loopless graph $G$ of order $n$ has a bipartite modulo $3$-regular subgraph if $|E(G)|> (3+\frac{1}{2})(n-1)$, or
$\chi(G)\le 4$ and $|E(G)|> 3(n-1)$.
}\end{cor}
Finally, we propose the following conjecture to introduce a sharp version of Corollary~\ref{cor:k=3}. Note also the graph $G$ obtained from the complete graph of order five by inserting a new copy of a triangle is a graph of size $(3+\frac{1}{4})(|V(G)|-1)$ having no bipartite modulo $3$-regular subgraph.
\begin{conj}
{A loopless graph $G$ of order $n$ has a bipartite modulo $3$-regular subgraph if $|E(G)|> (3+\frac{1}{4})(n-1)$, or $G$ is simple and $|E(G)|> 3(n-1)$. 
}\end{conj}
%
%
%
%
%
%
%
%
%

\end{document}